\pdfoutput=1

\documentclass[bj,authoryear,noshowframe]{imsart}

\usepackage{enumitem}
\RequirePackage{amsthm,amsmath,amsfonts,amssymb}
\usepackage{tikz}
\usetikzlibrary{arrows.meta, positioning, backgrounds, fit, calc,
                 decorations.pathreplacing, calligraphy}

\definecolor{nodecolor}{RGB}{60,100,160}
\definecolor{rootcolor}{RGB}{180,50,50}
\definecolor{edgecolor}{RGB}{60,60,60}

\allowdisplaybreaks[4]

\usepackage{xfrac} 
\usepackage{hyperref} 

\startlocaldefs
\theoremstyle{plain}

\newtheorem{theorem}{Theorem}[section]
\newtheorem{lemma}[theorem]{Lemma}
\newtheorem{corollary}[theorem]{Corollary}
\theoremstyle{definition}
\newtheorem{assumption}{Assumption}
\newtheorem{definition}[theorem]{Definition}
\newtheorem{remark}{Remark}

\newcommand{\Normal}{\text{N}}

\protected\def\[#1\]{\begin{equation}\begin{aligned}#1\end{aligned}\end{equation}}
\protected\def\(#1\){\begin{equation*}\begin{aligned}#1\end{aligned}\end{equation*}}

\endlocaldefs

\begin{document}

\begin{frontmatter}
\title{Consistency of Graphical Model-based Clustering: Robust Clustering using Bayesian Spanning Forest}
\runtitle{Consistency of Graphical Model-based Clustering}

\begin{aug}
\author[A]{\fnms{Yu}~\snm{Zheng}\ead[label=e1]{zheng.yu@ufl.edu}}
\author[A]{\fnms{Leo L.}~\snm{Duan}\ead[label=e3]{li.duan@ufl.edu}}
\and
\author[B]{\fnms{Arkaprava}~\snm{Roy}\ead[label=e2]{arkaprava.roy@ufl.edu}}

\address[A]{Department of Statistics, University of Florida\printead[presep={,\ }]{e1,e3}}

\address[B]{Department of Biostatistics, University of Florida\printead[presep={,\ }]{e2}}
\end{aug}

\begin{abstract}
Mixture model-based frameworks are very popular for statistical inference in clustering. While convenient for producing probabilistic estimates of cluster assignments and uncertainty, they are prone to misspecification, which can lead to inconsistent clustering results. Graphical model-based clustering adopts a different strategy, specifying the likelihood by treating data as dependently generated from a disjoint union of component graphs. Recent work on Bayesian spanning forests addresses graph uncertainty by using the integrated posterior of the node partition, marginalized over the latent edge distribution, to produce probabilistic clustering estimates. Despite strong empirical performance, theoretical guarantees such as consistency remain unclear, particularly when the true data-generating process deviates from the assumed graphical model. { This article establishes a positive asymptotic result: when data are generated from an unknown collection of component distributions and a mild asymptotic separation condition holds with probability tending to one (without requiring complete support separation), the posterior concentrates on the true partition, thereby yielding consistent clustering estimates, including the number of clusters.} Our results hold whether the number of clusters is fixed or increases with sample size. Additionally, we derive an upper bound on the expected misclassification rate. These results highlight graphical models as a robust alternative to mixture models in clustering.
\end{abstract}

\begin{keyword}
\kwd{Consistency of misspecified model}
\kwd{Gaussian oracle}
\kwd{Growing number of components}
\kwd{Misclassification rate}
\kwd{Node partition}
\kwd{Object-valued oracle}
\kwd{Refinement}
\end{keyword}

\end{frontmatter}

\section{Introduction}
Clustering aims to partition data into groups. To enable statistical inference on the clustering estimate, one assigns a generative model that involves a latent cluster assignment label for each data point, which leads to a probabilistic framework for characterizing clustering given the data. This model-based clustering framework is predominantly based on mixture models, where the data within a cluster are assumed to be independently and identically distributed (i.i.d.) from a component distribution, with many successful algorithms \citep{fraley2002model, zhong2003unified, baudry2010combining} further extending its popularity to what is seen today. On the theory front, mixture models are intimately related to the field of Bayesian nonparametric approaches, which assume the parameter space to be discrete and characterized by countably many choices. Infinite mixture models, such as the stick-breaking process mixture \citep{sethuraman1994constructive} and the mixture of finite mixtures \citep{miller2018mixture}, are a popular choice in this class, allowing the number of mixture components to be unbounded and the number of clusters (components with at least one data point assigned) to be estimated from the posterior distribution. As the number of data points increases, the posterior distribution of mixture models is shown to be consistent for estimating the ground-truth density governing the data-generating process, which is allowed to differ from the specified mixture model. Many consistent density estimation results have been established \citep{schwartz1965bayes, barron1999consistency, lijoi2005consistency, walker2007rates, nguyen2013convergence, petrone2014bayes}.

On the other hand, consistent estimation of clustering (the combinatorial partition of the data) is a much more challenging problem compared to density estimation. In particular, the conditions for establishing clustering consistency using mixture models are quite stringent regarding correct model specification. { Notably, when the distributional class of the component distributions is correctly specified, an overfitted mixture concentrates around the true mixture, thereby ensuring clustering consistency \citep{RousseauMengersen2011, VanHavreEtAl2015, GrazianRobert2018}.} In practice, however, correct specification faces difficulties on two fronts.

First, in specifying the mixing distribution that governs the component weights, \cite{miller2013simple, miller2014inconsistency} show that two popular infinite mixture models (Dirichlet process mixture and Pitman-Yor mixture) based on default choices of fixed hyperparameters lead to inconsistent estimates of the number of clusters, hence incorrect asymptotic targets for clustering. One remedy for achieving consistency is to calibrate the hyperparameters either deterministically based on the sample size \citep{ohn2023optimal, zeng2023consistent} or via a carefully chosen hyperprior \citep{Ascolani_2022}. Another remedy is to consider a finite mixture model with the number of components estimated either from the posterior distribution \citep{miller2018mixture} or using Bayes factors \citep{ishwaran2001bayesian, casella2014cluster, chib2016bayes, hairault2022evidence}.

Second, in specifying the component distribution that characterizes how data are conditionally i.i.d.\ within each cluster, misspecification is understandably common --- the mixture weights are a simple probability vector, but the choice of distribution family for each component is unlimited. \cite{miller2019robust} shows that misspecifying a skew-Gaussian mixture component as Gaussian leads to overestimation of the number of clusters, motivating their proposed power posterior to calibrate the effect of misspecification. \cite{cai2021finite} formally establishes the lack of robustness in the asymptotic regime, proving that slight model misspecification causes the posterior of finite mixture models to fail to concentrate on any finite number of clusters.

The risk of model misspecification and the lack of robustness inherent in mixture models clearly motivate the development of alternative likelihoods for clustering. Graphical models are an appealing choice. Specifically, one can imagine each cluster being associated with a likelihood based on a directed acyclic graph (DAG). By taking the union of these DAGs and assigning a prior distribution on their disjoint union, one obtains a generative model amenable to the canonical Bayesian paradigm for statistical inference. An early instance of this idea appears in single-linkage clustering, which yields consistent estimates albeit under a one-dimensional constraint \citep{hartigan1981consistency}. Single linkage is equivalent to restricting each DAG to be a tree whose undirected version is also acyclic. Although this restricts the family of DAGs, the impact on clustering is minor: from a clustering perspective, whether two points are connected directly or through a chain of edges, they belong to the same cluster regardless. The simplicity of tree graphs, together with efficient estimation algorithms, has motivated a growing body of work on unions of trees, also known as forests \citep{luo2021bayesian, luo2023nonstationary}.

With the above intuition, we note that the parameter of interest in clustering is the partition of nodes rather than the directed edges within each DAG. In this light, \cite{duan2023spectral} propose treating the edges in each DAG as latent variables and focusing on the integrated posterior with the edges marginalized out. Specifically, a Bayesian spanning forest model is used for graphical model-based clustering. Empirically, the point estimate of clustering is much improved compared to single-linkage clustering, owing to the Bayesian spanning forest model's explicit treatment of edge uncertainty. Theoretically, strong performance is explained by an asymptotic equivalence between the posterior mode (given the number of clusters) and the estimate produced by normalized spectral clustering \citep{ng2001spectral}, and by clustering consistency when data are generated from a forest graphical model. Nevertheless, it remains unknown whether the integrated posterior of the node partition is {\em robustly} consistent: if the data-generating mechanism differs from the specified graphical model, can the posterior still concentrate on a ground-truth partition for separable data points?

This article gives a positive answer for the Bayesian spanning forest model. When the data arise independently from unknown distributions given their labels, { the Bayesian spanning forest model can recover the ground-truth clustering under mild conditions. A key feature of our framework is its flexibility with respect to the true data-generating mechanism. We model the true density as a mixture whose component supports are only required to separate asymptotically, permitting arbitrary overlap at any finite $n$. No restrictions are placed on the functional form of the component densities, and the component distributions $G^{0,n}_k$ may differ across clusters in both shape and parametric family. Moreover, our main result (Theorem~\ref{thm:main}) does not require independence among observations within the same cluster; independence is introduced only in later sections to derive sharper results in concrete examples.}
To our knowledge, this is the first result in the model-based clustering literature showing that a potentially misspecified model can yield an asymptotically correct estimate. There are five key theoretical contributions. 
\begin{enumerate}
    \item Our findings demonstrate a feasible approach to bypassing the need for correct specification of the mixture component distribution.
    \item We show that the posterior enjoys consistency with {\em simultaneous} recovery of both the number of clusters and the true cluster labels; in the existing mixture model-based clustering literature, the latter is often achieved only under additional conditions, either restricting the family of data-generating distributions or assuming the number of clusters to be known.
    \item We develop the theory under very general conditions, allowing the number of true clusters to be fixed or diverging with the sample size, and further allowing the data dimension to grow with the sample size in certain more specific settings.
    \item When the number of clusters is assumed known, we provide a statistical upper bound on the misclassification rate.
    \item On the mathematical side, we develop a new refinement technique that may be of independent interest for the theoretical development of asymptotics in Bayesian clustering analysis.
\end{enumerate}

\section{Graphical model-based clustering}
{ We first introduce the notations, then describe the graphical model-based clustering framework and subsequently, derive the integrated posterior distribution under the Bayesian spanning forest model.}

\subsection{Notation}
We use $y_i$ to denote a data point in some metric space $\mathcal Y$. Let $[N]$ denote the data index set $\{1,2,\ldots,N\}$ for any positive integer $N$. The parameter of interest is a partition of $[n]$, $\mathcal V_n=(V_1,\ldots,V_K)$ associated with $K$ clusters: $\bigcup_{k=1}^KV_k=[n]$ and $V_k\bigcap V_{k'}=\emptyset$ for $k\not=k'$.
{ We denote the cardinality of $V_k$ by $|V_k|=n_k$.}
For a given partition $\mathcal V_n$, for any two points $s$ and $t$, we say $s\sim t$ under $\mathcal V_n$ if there exists some $k\in[K]$ such that $s,t\in V_k$; otherwise we say $s\not\sim t$ under $\mathcal V_n$.

{ Let $A\in[0,\infty)^{n\times n}$ be a matrix defined as $(A)_{ii}=0$ for $i\in[n]$ and $(A)_{ij}=f_{ij}=f(y_j|y_i;\theta)\ge0$, for $i\not=j$, where $f$ is some probability kernel to be specified later and $\theta$ is the set of parameters needed to characterize $f$.
Then $A_{V_k}$ denotes the submatrix of $A$ indexed by $V_k$ in both rows and columns.
For $A_{V_k}\in[0,\infty)^{n_k\times n_k}$, the Laplacian $L_{V_k}$ is a matrix of the same size, with $L_{V_k:i,j}=-A_{V_k:i,j}$ for $i\neq j$ and $L_{V_k:i,i}=\sum_{j\neq i}A_{V_k:i,j}$. In the case when $A$ is the weight matrix of a graph, $L_{V_k}$ is referred to as \emph{graph Laplacian} generated by the vertex set $V_k$ with weight matrix $A$. 
We write $L_{[n]}=:L_n$ for the Laplacian on $[n]$.
For a matrix $B$, we write $B[i]$ for the matrix obtained by removing the $i$-th row and column, and use $|\cdot|$ to denote the determinant.
For convenience, when the Laplacian is defined on a single point, we set $L_{\{i\}}=(0)$ and define $|L_{\{i\}}[1]|=1$. We use $J = \mathbf{1}\mathbf{1}^\top$ to denote the square matrix of all ones.}

For two sequences $\{a_n\}$ and $\{b_n\}$, we write $a_n\asymp b_n$ if there exist constants $C_1,C_2>0$, independent of $n$, such that $C_1b_n\leq a_n\leq C_2 b_n$ for sufficiently large $n$, and $a_n\precsim b_n$ if there exists a constant $C_1>0$ such that $a_n\leq C_1b_n$ for sufficiently large $n$. { In addition, we adopt the notation $a_n=o(b_n)$ to represent $\lim_{n\to\infty}a_n/b_n=0$.}

\subsection{Clustering with disjoint union of DAGs}

We now define the generative model used in the graphical model-based clustering framework. Associated with each cluster $V_k$, we consider a connected DAG, $O_k=(V_k,E_k, k^*)$ containing edges $E_k$, and root node $k^*\in V_k$.
We use $\mathcal E_{\mathcal V} = \{E_1,\ldots, E_K\}$ and $\mathcal R_{\mathcal V} = \{1^*,\ldots, K^*\}$ to denote the collections of edge sets and root nodes, respectively. Clearly, $(\mathcal V_n, \mathcal E_{\mathcal V_n}, \mathcal R_{\mathcal V_n})$ is a disjoint union of $K$ DAGs.

Consider the likelihood for data $y^{(n)}=\{y_1,\ldots,y_n\}$:
\[
P(y^{(n)}\mid \mathcal V_n, \mathcal E_{\mathcal V_n}, \mathcal R_{\mathcal V_n}, \theta) =  \prod_{k=1}^{K}
\bigg [ r(y_{k^*} ; \theta) \prod_{(i,j) \in E_k }f(y_i \mid y_j; \theta)
\bigg ],
\]
where each $r(\cdot;\theta)$ is the probability kernel of a {\em root} distribution that gives rise to the first data point in a cluster, and $f(\cdot\mid y_j;\theta)$ is the kernel of a {\em leaf} distribution that gives rise to a subsequent data point given an existing one in the cluster.
{
Here $\theta$ denotes the set of parameters characterizing these kernels. For example, a Gaussian kernel with shared variance $\sigma^2$ yields $f(y_i\mid y_j;\theta)\propto\exp(-d_{i,j}^2/(2\sigma^2))$, where $d_{i,j}$ is the distance between nodes and $\theta$ includes $\sigma^2$.
The kernel $r(\cdot;\theta)$ only characterizes the root node distribution and does not affect the within-cluster dependence structure. In particular, $r(\cdot;\theta)$ may be chosen as a diffuse density or even taken to be constant, since it need not integrate to one.
}

\subsection{Integrated posterior of clustering under Bayesian spanning forest model}

Let $\Pi_0(K,\mathcal V_n)$ be a partition probability function serving as the prior, and $\Pi_0(\mathcal E_{\mathcal V_n}, \mathcal R_{\mathcal V_n} \mid \mathcal V_n)$ the conditional prior for edges and roots.
Taking into account the formidably large combinatorial space of $\mathcal E_{\mathcal V_n}$ and $\mathcal R_{\mathcal V_n}$, we do not expect to have accurate estimates on these two parameters. Fortunately, our parameter of interest in clustering is $\mathcal V_{n}$ only, which can be characterized via the integrated posterior:
\[\label{eq:int_posterior_BSF}
\Pi^*(\mathcal V_n \mid y^{(n)}) =  \frac{
\sum_{\mathcal E_{\mathcal V_n}, \mathcal R_{\mathcal V_n}}
P(y^{(n)}\mid \mathcal V_n, \mathcal E_{\mathcal V_n}, \mathcal R_{\mathcal V_n}) \Pi_0(K,\mathcal V_n)\Pi_0(\mathcal E_{\mathcal V_n}, \mathcal R_{\mathcal V_n} \mid \mathcal V_n)
}{
\sum_{\mathcal V'_n, \mathcal E_{\mathcal V'_n}, \mathcal R_{\mathcal V'_n}}
P(y^{(n)}\mid \mathcal V'_n, \mathcal E_{\mathcal V'_n}, \mathcal R_{\mathcal V'_n}) \Pi_0(K,\mathcal V'_n)\Pi_0(\mathcal E_{\mathcal V'_n}, \mathcal R_{\mathcal V'_n} \mid \mathcal V'_n)
}.
\]

{
In this work, we consider spanning tree-based DAGs.
For a vertex set $V_k$ with $n_k$ nodes, a {\em spanning tree} is a connected acyclic graph whose vertex set is exactly $V_k$ and that contains $n_k-1$ edges.
When a root node $k^*\in V_k$ is specified and every edge is oriented away from the root so that each non-root node has exactly one parent, the resulting directed graph $O_k=(V_k,E_k,k^*)$ forms a {\em rooted spanning tree}, a special case of a DAG.
In this representation, each node in $V_k\setminus\{k^*\}$ has exactly one incoming edge from its parent, while the root has none.
Consequently, the disjoint union $(\mathcal V_n,\mathcal E_{\mathcal V_n},\mathcal R_{\mathcal V_n})$ forms a {\em spanning forest} consisting of $K$ rooted spanning trees corresponding to the clusters $V_1,\ldots,V_K$.
For clarity, Figure~\ref{fig:partition_forest} illustrates how a spanning forest is constructed from data and its corresponding partition.
We will show at the end of this subsection that the numerator in \eqref{eq:int_posterior_BSF} simplifies greatly when the summation is restricted to such spanning forests.
The Bayesian clustering model based on this structure is referred to as the {\em Bayesian spanning forest} (BSF) model \citep{duan2023spectral}.
}

\begin{figure}[ht]
\centering
\resizebox{\linewidth}{!}{
\begin{tikzpicture}[
    datanode/.style={circle, draw=nodecolor, fill=white, thick,
                     minimum size=7mm, inner sep=1pt, font=\scriptsize},
    clnode/.style={circle, draw=nodecolor, fill=white, thick,
                   minimum size=7mm, inner sep=1pt, font=\scriptsize},
    rootnode/.style={circle, draw=rootcolor, fill=white, very thick,
                     minimum size=7mm, inner sep=1pt, font=\scriptsize},
    cframe/.style={draw=black, thick, rounded corners=8pt,
                   inner sep=2pt, fill=none},
    >=Stealth,
    nlbl/.style={font=\small}
]

\begin{scope}[xshift=0cm]
  \node[datanode] (n1) at (0.4, 2.0) {$y_1$};
  \node[datanode] (n2) at (0.0, 0.9) {$y_2$};
  \node[datanode] (n3) at (0.8, 0.0) {$y_3$};
  \node[datanode] (n4) at (3.2, 2.6) {$y_4$};
  \node[datanode] (n5) at (3.8, 1.4) {$y_5$};
  \node[datanode] (n6) at (2.8, 0.8) {$y_6$};
  \node[datanode] (n7) at (0.4, 3.5) {$y_7$};
  \node[datanode] (n8) at (2.0, 3.7) {$y_8$};
  \node[font=\small] at (1.9, -1.2) {(a) Data $y^{(n)}$};
\end{scope}

\begin{scope}[xshift=5.8cm]
  \node[clnode] (p1) at (0.4, 2.0) {$y_1$};
  \node[clnode] (p2) at (0.0, 0.9) {$y_2$};
  \node[clnode] (p3) at (0.8, 0.0) {$y_3$};
  \node[clnode] (p4) at (3.2, 2.6) {$y_4$};
  \node[clnode] (p5) at (3.8, 1.4) {$y_5$};
  \node[clnode] (p6) at (2.8, 0.8) {$y_6$};
  \node[clnode] (p7) at (0.4, 3.5) {$y_7$};
  \node[clnode] (p8) at (2.0, 3.7) {$y_8$};

  \node[cframe, fit=(p1)(p2)(p3)] (fV1b) {};
  \node[cframe, fit=(p4)(p5)(p6)] (fV2b) {};
  \node[cframe, fit=(p7)(p8)]     (fV3b) {};

  \node[font=\footnotesize\bfseries, below=4pt]  at (fV1b.south)  {$V_1$};
  \node[font=\footnotesize\bfseries, below=4pt] at (fV2b.south)  {$V_2$};
  \node[font=\footnotesize\bfseries, above=4pt] at (fV3b.north) {$V_3$};

  \node[font=\small] at (1.9, -1.2) {(b) Partition $\mathcal{V}_n$};
\end{scope}

\begin{scope}[xshift=11.6cm]
  \node[rootnode] (t1) at (0.4, 2.0) {$y_1$};
  \node[clnode]   (t2) at (0.0, 0.9) {$y_2$};
  \node[clnode]   (t3) at (0.8, 0.0) {$y_3$};
  \node[rootnode] (t4) at (3.2, 2.6) {$y_4$};
  \node[clnode]   (t5) at (3.8, 1.4) {$y_5$};
  \node[clnode]   (t6) at (2.8, 0.8) {$y_6$};
  \node[rootnode] (t7) at (0.4, 3.5) {$y_7$};
  \node[clnode]   (t8) at (2.0, 3.7) {$y_8$};

  \draw[->, thick, edgecolor] (t1) -- (t2);
  \draw[->, thick, edgecolor] (t1) -- (t3);
  \draw[->, thick, edgecolor] (t4) -- (t6);
  \draw[->, thick, edgecolor] (t6) -- (t5);
  \draw[->, thick, edgecolor] (t7) -- (t8);

  \node[cframe, fit=(t1)(t2)(t3)] (fV1c) {};
  \node[cframe, fit=(t4)(t5)(t6)] (fV2c) {};
  \node[cframe, fit=(t7)(t8)]     (fV3c) {};

  \node[font=\footnotesize\bfseries, below=4pt]  at (fV1c.south)  {$V_1$};
  \node[font=\footnotesize\bfseries, below=4pt] at (fV2c.south)  {$V_2$};
  \node[font=\footnotesize\bfseries, above=4pt] at (fV3c.north) {$V_3$};

  \node[font=\small] at (1.9, -1.2) {(c) Spanning forest};
\end{scope}

\draw[->, thick, gray!65, line width=1.4pt] (4.5,  1.5) -- (5, 1.5);
\draw[->, thick, gray!65, line width=1.4pt] (10.3, 1.5) -- (10.8,1.5);

\end{tikzpicture}
}
\caption{
  From data to spanning forest.
  (a)~Eight data points with no assumed structure.
  (b)~A partition $\mathcal{V}_n = (V_1, V_2, V_3)$ groups nodes into
  three clusters (framed regions).
  (c)~One spanning forest compatible with $\mathcal{V}_n$: each cluster
  carries a rooted directed spanning tree, with edges oriented away from
  the root node $k^*$ (red). The edge sets $\mathcal{E}_{\mathcal{V}_n}$
  and root nodes $\mathcal{R}_{\mathcal{V}_n}$ are latent variables
  marginalized out of the integrated posterior~\eqref{eq:int_posterior_BSF}.
}
\label{fig:partition_forest}
\end{figure}

Since clustering is an unsupervised learning task aimed at grouping similar data points, the specific labels assigned to each $y_i$ are inconsequential. Accordingly, we introduce the following convention: two partitions $(V^1_1, \ldots, V^1_K)$ and $(V^2_1, \ldots, V^2_{K'})$ of $[n]$ are said to be \emph{equivalent}, denoted by $(V^1_1, \ldots, V^1_K) \sim (V^2_1, \ldots, V^2_{K'})$, if and only if $K = K'$ and there exists a bijection $\psi: [K] \to [K]$ such that $V^1_k = V^2_{\psi(k)}$ for all $k \in [K]$.

Under this equivalence relation, the space of partitions becomes a quotient space, where each equivalence class corresponds to a partition up to relabeling.
Given this structure, we define the integrated posterior probability of a partition (representing its entire equivalence class) as
\(
\Pi(\mathcal{V}_n \sim (V_1, \ldots, V_K) \mid y^{(n)}) &= \sum_{\mathcal{V}'_n : \mathcal{V}'_n \sim (V_1, \ldots, V_K)} \Pi^*(\mathcal{V}_n = \mathcal{V}'_n \mid y^{(n)}) \\&= K! \, \Pi^*(\mathcal{V}_n = (V_1, \ldots, V_K) \mid y^{(n)}),
\)
where the last equality holds since the equivalence class contains exactly $K!$ labelings, each assigned the same posterior probability.

For distinguishing purposes, throughout the paper, $\Pi(\mathcal{V}_n \mid y^{(n)})$ denotes the posterior probability of the equivalence class containing $\mathcal{V}_n$, whereas $\Pi^*(\mathcal{V}_n \mid y^{(n)})$ denotes the posterior probability that the partition is exactly equal to $\mathcal{V}_n$.

{
Following \cite{duan2023spectral}, we consider the prior
\(
\Pi_0(K,\mathcal V'_n)\Pi_0(\mathcal E_{\mathcal V'_n}, \mathcal R_{\mathcal V'_n} \mid \mathcal V'_n)
\propto
\lambda^{K}
1\!\left((\mathcal V'_n,\mathcal E_{\mathcal V'_n},\mathcal R_{\mathcal V'_n}) \text{ forms a spanning forest}\right)
\)
with some $\lambda>0$, together with a diffuse root kernel $r(\cdot)\equiv\delta$ for some $\delta>0$ and $K\in\{1,\ldots,n\}$. We recall the directed matrix-tree theorem, used to simplify the posterior probability $\Pi^*(\mathcal{V}_n \mid y^{(n)})$. An illustration on a three-node graph is provided in Section~S1.

\begin{lemma}[Directed matrix-tree theorem \citep{chaiken1978matrixtree}]
\label{lem:dir_mat_tree_thm}
Let $W=(w_{ij})$ be a weighted directed graph on vertex set $V$, and define the (possibly non-symmetric) graph Laplacian $L_V$ generated by the vertex set $V$ (with weight matrix $W$) as follow:
\(
L_{V:i,j}=
\begin{cases}
\sum_{m\neq i} w_{im}, & i=j,\\
-w_{ij}, & i\neq j .
\end{cases}
\)
Then for any $r_0\in V$,
\(
|L_V[r_0]|=\sum_{T\in\mathcal T_{r_0}}\prod_{(i,j)\in T} w_{ij},
\)
where $\mathcal T_{r_0}$ denotes the set of all directed spanning trees rooted at $r_0$.
\end{lemma}

Under the BSF model with root kernel $r(\cdot)\equiv\delta$, the integrated posterior can be written as
\[
\label{eq:post_prob_V_n_general}
\Pi^*(\mathcal V_n=(V_1,\ldots,V_K)|y^{(n)})
=&C_n\prod_{k=1}^K
\left[
\lambda
\sum_{r\in V_k}
\left(
r(y_r)
\sum_{\substack{E_k:\\ r_0\text{ is the root}}}
\prod_{(i,j)\in E_k} f(y_i\mid y_j;\theta)
\right)
\right] \\
=&C_n(\delta\lambda)^K
\prod_{k=1}^K
\left[
\sum_{r_0\in V_k}|L_{V_k}[r_0]|
\right],
\]
where $C_n$ is a normalizing constant, and $L_{V_k}$ is the graph Laplacian generated by the vertex set $V_k$. The second equality follows from Lemma~\ref{lem:dir_mat_tree_thm} with weights $w_{ij}=f(y_i\mid y_j;\theta)$. Furthermore, if the leaf kernel is symmetric, i.e.\ $f(y_i\mid y_j;\theta)=f(y_j\mid y_i;\theta)$, then the Laplacian becomes symmetric and all cofactors coincide.
Thus $|L[r_0]|$ equals any cofactor of $L$ and $\sum_{r_0\in V_k}|L_{V_k}[r_0]|=n_k|L_{V_k}[1]|$. This yields the classical Kirchhoff matrix-tree theorem; see, e.g., \cite{chaiken1978matrixtree}.
Substituting this identity into \eqref{eq:post_prob_V_n_general} yields
\[
\label{eq:post_prob_V_n}
\Pi^*(\mathcal V_n=(V_1,\ldots,V_K)|y^{(n)})=C_n(\delta\lambda)^K\prod_{k=1}^K\left[|L_{V_k}[1]|\cdot n_k\right]=C_n(\delta\lambda)^K\prod_{k=1}^K
\left|L_{V_k}+\frac{1}{n_k}J\right|,
\]
where the last equality follows from Remark~\ref{remark:det_Ln+J/n}, and $J=\mathbf{1}\mathbf{1}^\top$ is a conformable all-ones matrix.
Due to its computational simplicity and theoretical tractability, we focus on this symmetric-kernel setting in the remainder of the paper.

We now illustrate \eqref{eq:post_prob_V_n_general} and \eqref{eq:post_prob_V_n} with three elementary examples.
\begin{enumerate}
\item When each node forms its own cluster ($K=n$): there are no leaf kernel terms, and thus, the posterior reduces to $C_n(\delta\lambda)^n$.
\item When there is only one cluster ($K=1$): the posterior is $C_n\delta\lambda\sum_{i=1}^n|L_n[i]|$, where $L_n$ is the Laplacian formed by all $n$ nodes. When the leaf kernel is symmetric, all cofactors of $L_n$ coincide and thus, the posterior reduces to $C_n\delta\lambda n|L_n[1]|=C_n\delta\lambda|L_n+n^{-1}J|$.
\item When every two nodes form one cluster ($n=2K$): for each cluster, the contribution to the posterior is $r(y_1)f(y_2\mid y_1;\theta)+r(y_2)f(y_1\mid y_2;\theta)=\delta(|L_2[1]|+|L_2[2]|)$, since $r(\cdot)\equiv\delta$ and
\(
L_2=\begin{pmatrix}
f(y_1\mid y_2;\theta) & -f(y_1\mid y_2;\theta)\\
-f(y_2\mid y_1;\theta) & f(y_2\mid y_1;\theta)
\end{pmatrix},
\quad
L_2[1]=\begin{pmatrix}f(y_2\mid y_1;\theta)\end{pmatrix},
\quad
L_2[2]=\begin{pmatrix}f(y_1\mid y_2;\theta)\end{pmatrix}.
\)
The posterior is therefore $C_n(\delta\lambda)^{K}\prod_{k=1}^K\left[|L_{V_k}[1]|+|L_{V_k}[2]|\right]$. When the leaf kernel is symmetric, this reduces to $C_n(\delta\lambda)^K2^K\prod_{k=1}^K|L_{V_k}[1]|=C_n(\delta\lambda)^K\prod_{k=1}^K|L_{V_k}+2^{-1}J|$.
\end{enumerate}
}

\section{Problem setup and main results}\label{section:main}

\subsection{Label oracle}
{
Our goal is to establish posterior clustering consistency for the BSF model introduced in the previous section.
Posterior consistency studies whether the posterior distribution of a model parameter converges to a point mass at its true value as the number of observations tends to infinity, providing a frequentist calibration of the Bayesian procedure. Thus, there exists a true data-generating distribution with an underlying ground-truth parameter value for the observed data. This law may coincide with the fitted model (the correctly specified case) or differ from it (the misspecified case). Posterior consistency results for the BSF model in the correctly specified setting have been established in \cite{duan2023spectral}. Since that setting is somewhat restrictive, we focus here on the more general (misspecified) setting, where the truth is assumed to be a general mixture satisfying conditions described below.

Our target parameter is the clustering partition of the index set $[n]=\{1,\ldots,n\}$ for each sample size $n$.
We assume that there exists an underlying infinite clustering structure that is consistent across $n$.
Specifically, let
\(
(V_1^0, V_2^0, \ldots)
\)
be a partition of $\mathbb{N}^+$ such that
\(
V_k^0 \cap V_\ell^0 = \emptyset \quad (k \neq \ell),
\qquad
\bigcup_{k=1}^\infty V_k^0 = \mathbb{N}^+.
\)
For each $n \in \mathbb{N}^+$, define the induced partition of $[n]$ by
\(
V_k^{0,n} = V_k^0 \cap [n], \quad k=1,\ldots,K_{0,n},
\)
where $K_{0,n}$ denotes the number of nonempty clusters among the first $n$ indices. Note that $K_{0,n}$ may either remain bounded or diverge as $n\to\infty$.

This construction yields a nested sequence of partitions:
$(V_1^{0,n},\ldots,V_{K_{0,n}}^{0,n})$
is the restriction of
$(V_1^{0,n+1},\ldots,V_{K_{0,n+1}}^{0,n+1})$
to $[n]$. In particular, as $n$ increases, existing clusters persist and may grow, while new clusters may appear.

Let $z_i^*$ denote the cluster index such that $i \in V_{z_i^*}^0$.
Given the oracle partition, we assume that for each $n$,
\[ \label{oracle_data_generation} y_i \stackrel{\text{indep}}{\sim} G^{0,n}_{z_i^*}, \qquad i=1,\ldots,n, \]
where $\{G_k^{0,n}\}_{k=1}^{K_{0,n}}$ are cluster-specific distributions.

The distributions $G_k^{0,n}$ are allowed to depend on $n$, so observations across different sample sizes need not be nested.
Let $y^{(n)} := \{y_1,\ldots,y_n\}$ denote the sample of size $n$ generated under this scheme. We write $(\mathcal Y^{(n)},\mathcal F^{(n)},P_0^{(n)})$ for the probability space of $y^{(n)}$, where $P_0^{(n)}$ denotes the conditional distribution of $y^{(n)}$ given the oracle partition $(V_1^{0,n},\ldots,V_{K_{0,n}}^{0,n})$. More specifically, $P^{(n)}_0=\prod_{k=1}^{K_{0,n}}\prod_{i\in V_k}G^{0,n}_{z^*_i}$ with $P^{(n)}_0(\mathcal Y^{i-1}\times \mathcal B\times \mathcal Y^{n-i})=G^{0,n}_{z^*_i}(\mathcal B)$ for $\mathcal B\subset\mathcal Y$ in the $\sigma$-algebra of $G^{0,n}_{z^*_i}$; here $\times$ denotes the Cartesian product.
Since the distributions $G^{0,n}_k$ may vary with $n$, the sequence $y^{(n)}$ need not be nested.
Thus $y^{(n)}$ should not be viewed as a subsequence of $y^{(n+1)}$. Instead, we consider a triangular array of observations $y^{(1)},y^{(2)},\ldots,$ with $y^{(n)} = \{y_{n,1},\ldots,y_{n,n}\}$. For brevity, we write $y^{(\infty)} := \{y^{(1)},y^{(2)},\ldots\}$. When the sample size $n$ is fixed, we denote the elements of $y^{(n)}$ simply by $y_1,\ldots,y_n$, omitting the first index in $y_{n,i}$ to simplify notation.
}

\begin{remark}
From a frequentist perspective, when $n$ is given, the partition $(V_1^{0,n},\ldots,V_{K_{0,n}}^{0,n})$ (including $K_{0,n}$) should be viewed as fixed, with $y^{(n)}$ as the random variable.
\end{remark}

There are two main reasons for adopting the oracle specification in \eqref{oracle_data_generation}, in which data points within the same cluster are i.i.d.\ from some component distribution $G^0_k$.
First, it is natural to assume that observations from a given cluster are i.i.d.\ from a common distribution, even though the distributions $\{G^0_k\}_{k=1}^{K_0}$, as part of the oracle, are unknown and need not belong to any specific parametric family.
Second, under the BSF model, the integrated likelihood of the data in cluster $V_k$ is proportional to $|L_{V_k} + n_k^{-1} J|$, as shown in \eqref{eq:post_prob_V_n}, which is invariant to any permutation of data indices within $V_k$ for any $n_k \geq 1$. { When, in addition, the oracle distributions $\{G_k^{0,n}\}$ do not depend on $n$ for each fixed $k$,} this property is known as infinite exchangeability \citep{aldous1983exchangeability}.
By de Finetti's theorem \citep{hewitt1955symmetric, diaconis1980de, diaconis1980finite}, this implies that there exists some parameter $\zeta_k$ such that, conditional on $\zeta_k$, the observations $\{y_i : i \in V_k\}$ are i.i.d.\ from some distribution. Thus the BSF model specification is equivalent to an \emph{implicitly specified} finite mixture model.
To clarify, this does not mean that the implicit component distributions under the BSF model coincide with the oracle distributions $G^0_k$. We will nevertheless show that partition consistency can be achieved even without knowing whether the BSF model is correctly specified.

\subsection{Main results}\label{subsec:main_results}

We begin by formally defining posterior consistency for clustering.

\begin{definition}[Posterior consistency for clustering]
The posterior for the clustering $\mathcal V_n$ is said to be {\it consistent} at $(V_1^{0,\infty},V_2^{0,\infty},\ldots)$ if $\Pi(\mathcal V_n\not\sim(V_1^{0,n},\ldots,V_{K_{0,n}}^{0,n})|y^{(n)})\stackrel{n\to \infty}{\to} 0$ in $P^{(n)}_0$-probability.
\end{definition}

The following lemma, which follows immediately from the uniform boundedness of $\{\Pi(\mathcal V_n\not\sim(V_1^{0,n},\ldots,V_{K_{0,n}}^{0,n})|y^{(n)})\}_{n=1}^\infty$, gives an equivalent characterization of the definition.
\begin{lemma}
Clustering consistency is achieved at $(V_1^{0,\infty},V_2^{0,\infty},\ldots)$ if and only if
\(\mathbb E_{P^{(n)}_0}[\Pi(\mathcal V_n\not\sim(V_1^{0,n},\ldots,V_{K_{0,n}}^{0,n})|y^{(n)})]\stackrel{n\to \infty}{\to} 0.\)
\end{lemma}

Our results in this section develop general conditions for clustering consistency. Subsequently, Section~\ref{section:examples} illustrates their implications for specific oracle data-generating distributions.
Throughout, we make the following assumptions on the root kernel $r(\cdot)$ and the true number of clusters $K_{0,n}$.

\begin{assumption}
\label{ass1}
There exist constants $C_1, C_2 > 0$ such that, for sufficiently large $n$,
\(
C_1 \delta_n \leq \min_{i \in [n]} r(y_i) \leq \max_{i \in [n]} r(y_i) \leq C_2 \delta_n.
\)
\end{assumption}

\begin{assumption}
\label{ass2}
$K_{0,n}=o(\sqrt{n})$.
\end{assumption}

\begin{remark}
In Assumption~\ref{ass1}, $\delta_n$ depends on $n$. Specifically, we allow the root kernel to become more diffuse while remaining bounded away from zero as $n$ increases. In particular, the root kernel may be taken as the flat density $r(\cdot)\equiv\delta_n\equiv 1$.
\end{remark}

\begin{remark}
Assumption~\ref{ass2} imposes a growth condition on the true number of clusters $K_{0,n}$. This constraint is not imposed in the BSF model specification, which does not incorporate any prior knowledge about $K_{0,n}$.
\end{remark}

In the following, we write $f^{(n)}_{st}=f(y_t|y_s;\theta_n)$ for the conditional probability kernel between data points $y_s$ and $y_t$, referred to hereafter as the conditional kernel. The value of $f^{(n)}_{st}$ quantifies the probabilistic similarity between two data points. The key step in establishing our results is to control these conditional kernels efficiently via $\theta_n$. Dependence of $f_{st}^{(n)}$ on $n$ is not uncommon in large-sample analysis. For example, when $f_{st}^{(n)}=(\sqrt{2\pi}\sigma_n)^{-p}\exp\left\{-\|y_s-y_t\|_2^2/(2\sigma_n^2)\right\}$ is a Gaussian kernel, the decay rate of $\sigma_n$ controls the strength of dependence in $f_{st}^{(n)}$. Such sample-size- and dimension-dependent hyperparameter specifications are standard in asymptotic analysis of statistical methods \citep{castillo2015bayesian}. With controls on $(f_{st}^{(n)},\delta_n,\lambda_n)$, we define the following set $\mathcal D^{(\infty)}$:
\(
\begin{aligned}
\mathcal D^{(\infty)}:=\bigg\{y^{(\infty)}:&\frac{\sup_{s\not\sim t;s,t\in[n]}f^{(n)}_{st}}{\delta_n\lambda_n}\precsim (K_{0,n}-1+\iota_1)^{-n}\text{ for a fixed constant }\iota_1>0;\\
&\frac{\delta_n\lambda_n}{\inf_{s'\sim t';s',t'\in[n]}f^{(n)}_{s't'}}\precsim (K_{0,n}+1+\iota_2)^{-n}\text{ for a fixed constant }\iota_2>0\bigg\}.
\end{aligned}
\)
We clarify that the conditions defining $\mathcal{D}^{(\infty)}$ hold for sufficiently large $n$, that $s\sim t$ is understood under the oracle clustering $(V^{0,\infty}_1,V^{0,\infty}_2,\ldots)$, and that different values of $n$ correspond to different instances of $(y^{(n)}, f^{(n)}_{st}, \delta_n, \lambda_n, K_{0,n})$.

The set $\mathcal D^{(\infty)}$ collects sequences of data points that asymptotically satisfy inequalities shown to be essential for clustering consistency; it may therefore be regarded as a {\em nice} set.

\begin{remark}
The conditions defining $\mathcal D^{(\infty)}$ require both $\sfrac{\sup_{s\not\sim t;s,t\in[n]}f^{(n)}_{st}}{\delta_n\lambda_n}$ and $\sfrac{\delta_n\lambda_n}{\inf_{s'\sim t';s',t'\in[n]}f^{(n)}_{s't'}}$ to decay at least exponentially fast.
Intuitively, this exponential control stems from the combinatorial complexity of the partition space: the number of ways to partition $n$ data points into $K$ nonempty clusters grows as $O(K^n)$. For the posterior to concentrate on the oracle partition, the between-cluster kernel values must decay exponentially, while the within-cluster kernel values must grow exponentially, relative to the scaling parameter $\delta_n\lambda_n$.
\end{remark}

The following lemma shows that, under mild conditions, $\mathcal D^{(\infty)}$ is a subset of $\{y^{(\infty)}:\Pi(\mathcal V_n\not\sim(V_{1}^{0,n},\ldots,V_{K_{0,n}}^{0,n})|y^{(n)})=o(1)\}$, which is the key condition for our subsequent posterior consistency result. In other words, on $\mathcal D^{(\infty)}$, the posterior of $\mathcal{V}_n$ concentrates strongly around the true partition.

\begin{lemma}
\label{lem:subset}
Suppose Assumptions~\ref{ass1} and \ref{ass2} hold. Then
\(
\mathcal D^{(\infty)}\subset \{y^{(\infty)}:\Pi(\mathcal V_n\not\sim(V_{1}^{0,n},\ldots,V_{K_{0,n}}^{0,n})|y^{(n)})=o(1)\}.
\)
\end{lemma}

For a more concrete characterization of $\mathcal D^{(\infty)}$, define, for any positive constants $(c_1,c_2,\iota_1,\iota_2)\stackrel{\Delta}{=}\phi$:
\(
\begin{aligned}
\mathcal D^{(n)}_\phi:=\bigg\{y^{(n)}:&\frac{\max_{s\not\sim t;s,t\in[n]}f^{(n)}_{st}}{\delta_n\lambda_n}\leq c_1(K_{0,n}-1+\iota_1)^{-n},\\
&\frac{\delta_n\lambda_n}{\min_{s'\sim t';s',t'\in[n]}f^{(n)}_{s't'}}\leq c_2(K_{0,n}+1+\iota_2)^{-n}\bigg\}.
\end{aligned}
\)

\begin{theorem}[General clustering consistency under BSF]
\label{thm:main}
Suppose Assumptions~\ref{ass1} and \ref{ass2} hold. Then
\(
\mathbb E_{P^{(n)}_0}[\Pi(\mathcal V_n\not\sim(V_1^{0,n},\ldots,V_{K_{0,n}}^{0,n})|y^{(n)})]\stackrel{n\to \infty}{\to} 0,
\)
if $P_0^{(n)}(y^{(n)}\not\in \mathcal D^{(n)}_\phi)\stackrel{n\to\infty}{\to}0$ for a fixed constant $\phi$.
\end{theorem}

\begin{proof}[Proof of Theorem~\ref{thm:main}]
For simplicity, let $Z_n:=\Pi(\mathcal V_n\not\sim(V_1^{0,n},\ldots,V_{K_{0,n}}^{0,n})|y^{(n)})\in[0,1]$. We have the decomposition
\(
\mathbb E_{P^{(n)}_0}[Z_n]=\mathbb E_{P^{(n)}_0}[Z_n1_{\mathcal D^{(n)}_\phi}] + \mathbb E_{P^{(n)}_0}[Z_n1_{(\mathcal D^{(n)}_\phi)^c}],
\)
where $1_A$ denotes the indicator function of the set $A$.

For any $\epsilon_0>0$, there exists $N_1\in\mathbb N^+$ such that for any $n>N_1$,
\(
0\le\mathbb E_{P^{(n)}_0}[Z_n1_{(\mathcal D^{(n)}_\phi)^c}]\le P_0^{(n)}(y^{(n)}\not\in \mathcal D^{(n)}_\phi)<\epsilon_0/2.
\)
For the first term, invoking Lemma~\ref{lem:subset} on any sequence $y^{(\infty)}$ satisfying $y^{(n)}\in\mathcal D^{(n)}_\phi$ for $n>N_1$, there exists $N_2\in\mathbb N^+$ such that for any $n>\max(N_1,N_2)$,
\(
0\leq \mathbb E_{P^{(n)}_0}[Z_n1_{\mathcal D^{(n)}_\phi}]< \mathbb E_{P^{(n)}_0}[(\epsilon_0/2)1_{\mathcal D_\phi^{(n)}}]\le\epsilon_0/2.
\)
Combining these inequalities completes the proof.
\end{proof}

\begin{remark}
{ From a generative standpoint, conditional on the partition $\mathcal V_n$, the parameter $\delta_n$ governs the selection of the root node for each tree. Under a flat kernel, no data point is preferred as the root in constructing a directed tree within a cluster. This is part of the likelihood and is independent of the prior specification. In contrast, $\lambda_n$ enters through the prior. Since $\delta_n$ and $\lambda_n$ always appear together as a product, one may treat $\delta_n\lambda_n$ as a single hyperparameter controlling the clustering behavior of the BSF model. As $\delta_n$ becomes flatter, $\lambda_n$ may be taken larger, thereby permitting a less informative prior.}
Moreover, fixing either $\delta_n$ or $\lambda_n$ does not affect the consistency result.
For example, one can fix $r(\cdot)\equiv1$ and control $\lambda_n$ so that
\begin{itemize}
    \item ${\sup_{s\not\sim t;s,t\in[n]}f^{(n)}_{st}}/{\lambda_n}\precsim (K_{0,n}-1+\iota_1)^{-n}$ for a fixed constant $\iota_1>0$;
    \item ${\lambda_n}/{\inf_{s'\sim t';s',t'\in[n]}f^{(n)}_{s't'}}\precsim (K_{0,n}+1+\iota_2)^{-n}$ for a fixed constant $\iota_2>0$.
\end{itemize}
\end{remark}

\begin{remark}
When $K_{0,n}\equiv1$, i.e., the true number of clusters is always one, $\mathcal D^{(n)}_\phi$ takes the simpler form
\(
\mathcal D^{(n)}_\phi=\left\{y^{(n)}:\frac{\delta_n\lambda_n}{\min_{s'\sim t';s',t'\in[n]}f^{(n)}_{s't'}}\leq c_2(2+\iota_2)^{-n}\right\}.
\)
\end{remark}

{
\begin{remark}
Theorem~\ref{thm:main} is a general result. The key condition $P_0^{(n)}(y^{(n)}\not\in \mathcal D^{(n)}_\phi)\stackrel{n\to\infty}{\to}0$ requires only that the data arise from $K_{0,n}$ distinct subgroups, and does not impose independence within clusters. Hence the i.i.d.\ data-generation scheme in \eqref{oracle_data_generation} is not necessary for the validity of Theorem~\ref{thm:main}. We invoke this scheme in the next section to illustrate the main result through concrete examples.
\end{remark}}

\subsection{Gaussian-BSF}
We use the Gaussian-BSF model to elucidate the conditions on $\mathcal{D}^{(n)}_\phi$. For data lying in a $p$-dimensional Euclidean space, we define $d(y_s, y_t) = \|y_s - y_t\|_2$ and $f^{(n)}{st} = (\sqrt{2\pi},\sigma_n)^{-p} \exp\left\{ -\frac{d(y_s, y_t)^2}{2\sigma_n^2} \right\}$. More generally, when the data lie in a metric space $(\mathcal{Y}, d)$ with $\mathcal{Y} \subseteq \mathcal{M}$, where $(\mathcal{M}, g)$ is a homogeneous Riemannian manifold, one may instead employ a Riemannian Gaussian kernel. The Riemannian metric $g$ assigns to each point $x \in \mathcal{M}$ a symmetric positive-definite bilinear form $g_x : \mathcal{T}_x \mathcal{M} \times \mathcal{T}_x \mathcal{M} \to \mathbb{R}$ on the tangent space $\mathcal{T}_x \mathcal{M}$. In this setting, the conditional kernel takes the form $f^{(n)}_{st}=\zeta(\sigma_n)\exp\left\{-d_g(y_s,y_t)^2/(2\sigma_n^2)\right\}$,
where $d_g$ denotes the geodesic distance induced by $g$ and $\zeta(\sigma_n)$ is a normalizing constant that does not depend on the conditioning mean, owing to the homogeneity of $\mathcal{M}$ (see \cite{chakraborty2019statistics,said2022gaussian} for details).
For notational convenience, we write $d_{st} := d_g(y_s, y_t)$. In specific applications below, we will require the metrics $d(\cdot, \cdot)$ and $d_g(\cdot, \cdot)$ to be equivalent, as formalized in Assumption~\ref{assmp:mani}.

We emphasize that the oracle distributions $G^0_k$s' are not assumed to be Gaussian or any specific form; the results hold even when the oracle distributions are discrete.
Plugging the specific form of $f_{st}^{(n)}$ for Gaussian-BSF in the conditions of $\mathcal D^{(n)}_\phi$ and rearranging terms, we obtain
\[
\label{eq:D_set_for_Gaussian_BSF}
\mathcal D^{(n)}_\phi=\left\{y^{(n)}:\min_{s\not\sim t;s,t\in[n]}d^2_{st}\geq a_n,\max_{s'\sim t';s',t'\in[n]}d^2_{s't'}\leq b_n\right\},
\]
where
\[
\label{eq:a_n_b_n}
\begin{cases}
a_n=2\sigma_n^2\left[n\log(K_0-1+\iota_1)-\log(\delta_n\lambda_n)+\log(\zeta(\sigma_n))-\log(c_1)\right],\\
b_n=2\sigma_n^2\left[-n\log(K_0+1+\iota_2)-\log(\delta_n\lambda_n)+\log(\zeta(\sigma_n))+\log(c_2)\right].
\end{cases}
\]
For Euclidean distance, $\log(\zeta(\sigma_n))=-p\log(\sqrt{2\pi}\sigma_n)$. \cite{said2022gaussian} give expressions for $\zeta(\sigma_n)$ for a wide class of homogeneous Riemannian manifolds.

\begin{remark}[Interpretation of the condition on the oracle]
\label{rem:oraclecondi}
The conditions in \eqref{eq:D_set_for_Gaussian_BSF} admit a straightforward interpretation.
Each sample of size $n$ in $\mathcal D^{(n)}_\phi$ satisfies:
\begin{enumerate}
\item The minimum distance between any two points from different oracle clusters is bounded below by $a_n$, ensuring adequate separation between clusters.
\item The maximum distance between any two points within the same oracle cluster is bounded above by $b_n$, ensuring that within-cluster points remain in a compact region.
\end{enumerate}
These conditions need not hold for every sample of size $n$, but the probability that they hold must tend to one as $n\to\infty$, as required by Theorem~\ref{thm:main}.
\end{remark}

\begin{remark}
Some form of separation condition is necessary for clustering consistency.
{ For instance, \cite{loffler2021optimality} also highlight that the minimum cluster separation must diverge with $n$ for consistency to hold in spectral clustering.}
In that regard, our conditions are relatively mild: they would not suffice to guarantee consistency in conventional clustering models. For example, in widely used infinite mixture models, even slight model misspecification can lead the posterior to overestimate the number of clusters \citep{cai2021finite}, because splitting a tightly grouped cluster can increase the posterior probability under such models, making overpartitioning favorable.
In contrast, the BSF model resists overpartitioning through two structural properties.
First, splitting a cluster requires forming two new trees, which is penalized by the factor $\delta_n\lambda_n$.
Second, the marginal posterior in \eqref{eq:post_prob_V_n} involves the determinant $|L_{V_i} + n_i^{-1} J|$; splitting $V_i$ into $V_{i1}$ and $V_{i2}$ replaces this with the product $|L_{V_{i1}} + n_{i1}^{-1} J|\cdot|L_{V_{i2}} + n_{i2}^{-1} J|$, which is typically smaller when the points in $V_i$ are close together, making the split unfavorable.
Together, these properties ensure that the BSF model naturally resists overpartitioning under relatively mild separation conditions. A concrete illustration is provided in Remark~\ref{remark:miller_example}.
{ Furthermore, no specific functional form is imposed on the component densities, and the components $\{G^{0,n}_k\}$ need not belong to the same distributional family. Our main result does not require independence among observations within a given cluster; that assumption is invoked only later to simplify the conditions in Theorem~\ref{thm:main2}.}
\end{remark}

{ Theorem~\ref{thm:main2} below applies the general result of Theorem~\ref{thm:main} to the Gaussian-BSF model, deriving explicit, verifiable conditions in the spirit of Remark~\ref{rem:oraclecondi}.
In this theorem and thereafter, for any $k,\ell\in[K_{0}]$, we use $D_{k\ell}$ to denote the random variable representing the distance between two observations with cluster labels $k$ and $\ell$, respectively. We assume that, conditional on $(V^0_1,V^0_2,\ldots)$, the joint distribution of any pair $(y_s,y_t)$ depends only on their cluster labels $(z^*_s, z^*_t)$ and not on the specific indices $(s,t)$; this holds trivially under \eqref{oracle_data_generation}.}

\begin{theorem}[Clustering consistency under Gaussian-BSF model]
\label{thm:main2}
Suppose Assumptions~\ref{ass1} and \ref{ass2} hold. Then for $f^{(n)}_{st}=\zeta(\sigma_n)\exp\left\{-d^2_{st}/2\sigma_n^2\right\}$,
\(
\mathbb E_{P^{(n)}_0}[\Pi(\mathcal V_n\not\sim(V_1^{0,n},\ldots,V_{K_{0,n}}^{0,n})|y^{(n)})]\stackrel{n\to \infty}{\to} 0,
\)
if there exists $\phi\in\mathbb R^4_+$ such that
\[\label{cond:suf1}
\sup_{k\not=\ell;k,\ell\in[K_{0,n}]}P(D_{k\ell}^2<a_n)=o(1/n^2),
\]
and
\[\label{cond:suf2}
\sup_{k'\in[K_{0,n}]}P(D^2_{k'k'}>b_n)=o(1/n^2),
\]
where $a_n$ and $b_n$ are as in \eqref{eq:a_n_b_n}.
\end{theorem}

The proof proceeds by calibrating the growth and decay rates of $a_n$ and $b_n$ under the oracle distribution so as to satisfy \eqref{cond:suf1} and \eqref{cond:suf2}. The complete proof is in the Supplementary. Section~\ref{section:examples} provides concrete examples illustrating that relatively mild conditions on these sequences are sufficient to ensure clustering consistency under the Gaussian-BSF model.

\subsection{Misclassification rate with known $K_{0,n}$}
The consistency results above do not assume $K_{0,n}$ to be known.
However, if $K_{0,n}$ is known, one can further quantify the misclassification error rate and study its large-sample properties under the same assumptions, while still allowing it to grow.
Known $K_{0,n}$ is standard in the misclassification literature \citep{loffler2021optimality,chen2024achieving}.

For two partitions $\mathcal V^1_n = (V^1_1, \ldots, V^1_{K_{0,n}})$ and $\mathcal V^2_n = (V^2_1, \ldots, V^2_{K_{0,n}})$ of $n$ nodes, let $z^1$ and $z^2$ denote the corresponding label vectors.
Let $d_H(\cdot,\cdot)$ denote the permutation-invariant Hamming distance \citep{zhang2016minimax}, defined by $d_H(\mathcal V^1_n,\mathcal V^2_n)=\min_{\psi\in\Psi}\sum_{i=1}^n\mathbf{1}\{\psi(z^1_{i})\neq z^2_{i}\}$, where $\Psi= \{\psi : \psi \text{ is a bijection from } [K_{0,n}] \text{ to } [K_{0,n}]\}$. The following lemma bounds the posterior expected misclassification rate in terms of $f^{(n)}_{st}$ and $d_{st}$.

\begin{lemma}
\label{lem:mis_rate}
Suppose $\varepsilon_n/\gamma_n:=\sup_{s\not\sim t;s,t\in[n]}f^{(n)}_{st}/\inf_{s'\sim t';s',t'\in[n]}f^{(n)}_{s't'}=o(1/n)$. Suppose that Assumption~\ref{ass1} holds and $K_{0,n}$ is known. Then
\(
\mathbb{E}(d_H(\mathcal V_n,(V_1^{0,n},\ldots,V_{K_{0,n}}^{0,n}))\mid y^{(n)}) \precsim \exp\left\{\log\left(\frac{\varepsilon_n}{\gamma_n}\right)+n\log(K_{0,n}+1)\right\}.
\)
For the Gaussian-BSF model with $f^{(n)}_{st}=\zeta(\sigma_n)\exp\left\{-d^2_{st}/2\sigma_n^2\right\}$,
\(
\mathbb{E}(d_H(\mathcal V_n,(V_1^{0,n},\ldots,V_{K_{0,n}}^{0,n}))\mid y^{(n)}) \precsim \exp\bigg\{&-\frac{\inf_{s\not\sim t;s,t\in[n]}d^2_{st}-\sup_{s'\sim t';s',t'\in[n]}d^2_{s't'}}{2\sigma_n^2}\\
&\quad+n\log(K_{0,n}+1)\bigg\}.
\)
\end{lemma}

Finally, the next result bounds the expected misclassification rate. Specifically, it follows from Lemma~\ref{lem:mis_rate}. Its implications are discussed in Section~\ref{sec:rateknownK} through a concrete example.

\begin{theorem}
\label{thm:expected_misclass_rate_bound}
Suppose that Assumption~\ref{ass1} holds and $K_{0,n}$ is known. Under the Gaussian-BSF model with $f^{(n)}_{st}=\zeta(\sigma_n)\exp\left\{-d^2_{st}/2\sigma_n^2\right\}$,
\(
&\mathbb E_{P^{(n)}_0}[\mathbb{E}(d_H(\mathcal V_n,(V_1^{0,n},\ldots,V_{K_{0,n}}^{0,n}))\mid y^{(n)})]\\\precsim& \exp\left\{-\frac{a'_n-b'_n}{2\sigma_n^2}+n\log(K_{0,n}+1)\right\}\\&\quad+n^3\left[\sup_{k\not=\ell;k,\ell\in[K_{0,n}]}P(D^2_{k\ell}<a'_n)+\sup_{k'\in[K_{0,n}]}P(D^2_{k'k'}>b'_n)\right]
\)
for any positive sequences $\{a'_n\}$ and $\{b'_n\}$ satisfying $\sigma_n^2\log(n)/(a'_n-b'_n)=o(1)$.
\end{theorem}

\section{Concrete examples}\label{section:examples}

We first present examples illustrating the general clustering consistency result of the previous section in Section~\ref{sec:unknownK0}, allowing the true number of clusters $K_{0,n}$ to be unknown and growing. We then quantify the misclassification rate under known $K_{0,n}$ in Section~\ref{sec:rateknownK}, enabling comparison with the existing literature.

\subsection{Clustering consistency with unknown $K_{0,n}$}
\label{sec:unknownK0}

The first assumes that the oracle distributions $G^0_{k}$ are Gaussian; the second generalizes to object-valued distributions supported on a metric space satisfying Assumption~\ref{assmp:mani}. Both examples apply Theorem~\ref{thm:main2}.
In the first case, $G^0_k$s' are distinct Gaussian distributions. Here $\zeta(\sigma_n)=(\sqrt{2\pi}\sigma_n)^{-p}$, and the two probabilities in \eqref{cond:suf1} and \eqref{cond:suf2} can be bounded directly using properties of Gaussian and chi-squared distributions. The following lemma is proved in \cite{ghosh2021tailbounds}, Theorem~1.

\begin{lemma}
\label{lem:tailbound}
Suppose $X\sim\chi^2_p$. Then for $a>p$,
\(
P(X>a)\leq\exp\left\{-\frac{p}{2}\left[\frac{a}{p}-1-\log\!\left(\frac{a}{p}\right)\right]\right\}.
\)
\end{lemma}

We now state the clustering consistency result for the Gaussian-BSF model. The ambient dimension, $p_n$, is allowed to grow slowly with $n$, and parameters associated with the oracle distributions may also vary with $n$, though we suppress this dependence in the notation for brevity.

\begin{theorem}[Consistency when oracle distributions are Gaussian]
\label{thm:GMM_consistency_gen}
Suppose $(V^{0,\infty}_1,V^{0,\infty}_2,\ldots)$ is the oracle clustering for $y^{(\infty)}$, and $f^{(n)}_{st}=(\sqrt{2\pi}\sigma_n)^{-p_n}\exp\left\{-\|y_s-y_t\|_2^2/(2\sigma_n^2)\right\}$.
For each $n$, suppose $y_i\stackrel{\mathrm{indep}}{\sim}\mathrm{N}(\mu_{k},\Sigma_{k})$ if $y_{i}\in V^{0,\infty}_k$.
Set $\rho_n:=(\delta_n\lambda_n\sigma_n^{p_n})^{-1}$, $\Lambda_{\max}:=\max_{k\in[K_{0,n}]}\lambda_{\max}(\Sigma_k)$, and $D_{\mu,{\min}}:=\min_{k,\ell\in[K_{0,n}],k\not=\ell}\|\mu_k-\mu_\ell\|_2$.
Assume that
\begin{enumerate}
\item[(i)] Assumptions~\ref{ass1} and \ref{ass2} hold;
\item[(ii)] $\rho_n\succsim (K_{0,n}+1+\iota)^{n}$ for a fixed constant $\iota>0$;
\item[(iii)] $\sigma_n^2\log(\rho_n)/D^2_{\mu,\min}=o(1)$;
\item[(iv)] $\Lambda_{\max}(p_n\vee\log(n))/[\sigma_n^2\log(\rho_n)]=o(1)$.
\end{enumerate}
Then
\(\mathbb E_{P^{(n)}_0}[\Pi(\mathcal V_n\not\sim(V_1^{0,n},\ldots,V_{K_{0,n}}^{0,n})|y^{(n)})]\stackrel{n\to \infty}{\to} 0.\)
\end{theorem}

Here $\lambda_{\max}(\Sigma_k)$ denotes the largest eigenvalue of $\Sigma_k$.
To illustrate the assumptions on the oracle, we present the following corollary. Define the signal-to-noise ratio as $\mathrm{SNR}:=D_{\mu,\min}/\sqrt{\Lambda_{\max}}$.

\begin{corollary}
\label{cor:GMM_consistency_gen}
Suppose Assumptions~\ref{ass1} and \ref{ass2} hold. If $\mathrm{SNR}/\sqrt{p_n\vee\log(n)}\to\infty$ as $n\to\infty$, then there exists $(\delta_n\lambda_n,\sigma_n^2)$ yielding clustering consistency under the Gaussian-BSF model.
\end{corollary}

\begin{proof}[Proof of Corollary~\ref{cor:GMM_consistency_gen}]
Taking $-\sigma_n^2\log(\delta_n\lambda_n\sigma_n^{p_n})\asymp(D^2_{\mu,\min})^\alpha(\Lambda_{\max}(p_n\vee\log(n)))^{1-\alpha}$ for any $\alpha\in(0,1)$, one can verify that conditions~(iii) and (iv) of Theorem~\ref{thm:GMM_consistency_gen} are satisfied. Condition~(ii) then holds for many choices of $\delta_n\lambda_n$ and $\sigma_n$ consistent with the growth and decay rates of $D_{\mu,\min}$ and $\Lambda_{\max}$. One specific choice is
\(
&\sigma_n^2=\left[\mathrm{SNR}/\sqrt{p_n\vee\log(n)}\right]^{2\alpha}\cdot\Lambda_{\max}\log(n)/[n\log(K_{0,n}+1+\iota)],\\
&\delta_n\lambda_n=(K_{0,n}+1+\iota)^{-n}\sigma_n^{-p_n}.
\)
\end{proof}

Condition~(ii) of Theorem~\ref{thm:GMM_consistency_gen} quantifies the required rate for the product $\delta_n\lambda_n\sigma_n^{p_n}$. Given conditions~(iii) and (iv), a wide range of choices for $(\delta_n\lambda_n,\sigma_n)$ is available, depending on the minimum mean separation $D_{\mu,\min}$, the maximum spread $\Lambda_{\max}$, and the pair $(p_n,K_{0,n})$. To illustrate these conditions concretely via Corollary~\ref{cor:GMM_consistency_gen}, consider the special case $D^2_{\mu,\min}\asymp n^d$, $\Lambda_{\max}\asymp n^h$, $p_n\asymp \log(n)$, and $K_{0,n}\asymp n^{\kappa}$ with $\kappa<\sfrac{1}{2}$, so that Assumption~\ref{ass2} holds.
Corollary~\ref{cor:GMM_consistency_gen} requires $n^{d-h}/\log(n)\to\infty$, i.e., $d>h$.
One may then take $\sigma_n^2\asymp n^{\beta-1}$ for $\beta\in(h,d)$ and $\delta_n\lambda_n\asymp n^{-n}$ to achieve clustering consistency. We discuss three cases depending on the values of $d$ and $h$.

\medskip
\noindent\textit{Case 1} ($h\ge1$): Both the minimum separation and the maximum spread increase rapidly with $n$. To accommodate the increasing within-cluster spread, $\sigma_n^2$ must also grow with $n$; otherwise the BSF model may incorrectly split a single cluster into multiple subclusters.

\medskip
\noindent\textit{Case 2} ($d>1$ and $h<1$): The minimum separation grows with $n$, while the maximum spread grows slowly or decreases. In this regime $\sigma_n^2$ has greater flexibility, with $\beta-1\in(h-1,d-1)$; accordingly $\sigma_n^2$ may increase, remain constant, or decrease with $n$.

\medskip
\noindent\textit{Case 3} ($d\le1$): Both the minimum separation and maximum spread grow slowly or shrink with $n$. Clusters may not be sufficiently separated, risking erroneous merging of distinct clusters. To prevent underpartitioning, $\sigma_n^2$ must decrease with $n$, making the BSF model more sensitive to small separations.
\medskip

The above discussion provides insight into the roles of $D_{\mu,\min}$ and $\Lambda_{\max}$ as the data dimension $p_n$ and the number of clusters $K_{0,n}$ grow. { Without additional structural conditions, our clustering consistency result accommodates data dimensions growing moderately with the sample size, as ensured by condition~(iv) of Theorem~\ref{thm:GMM_consistency_gen}.}

\begin{remark}\label{remark:miller_example}
\cite{miller2013simple} presents a toy example showing that Dirichlet process mixtures (DPMs) with fixed hyperparameters can fail to recover the true number of clusters asymptotically: the posterior may favor multiple clusters even when the data arise from a single one. { This inconsistency was later addressed by \cite{Ascolani_2022}, who showed in Proposition~2 of their paper that degeneracy at zero in the posterior of the concentration parameter $\alpha$ is a necessary condition for clustering consistency, and who resolved the issue by placing a prior on $\alpha$. Their approach, however, requires stronger controls on the true component densities than those needed here.} Consider the same toy example from \cite{miller2013simple}, where all observations are i.i.d.\ from $\mathrm{N}(0,1)$. Applying Theorem~\ref{thm:GMM_consistency_gen}, clustering consistency follows under the specification
\(
r(\cdot) \equiv 1, \quad \sigma_n^2 \equiv 1, \quad \lambda_n \asymp 3^{-n}.
\)
\end{remark}

We next generalize to settings where the data support is non-Euclidean. The robustness of the BSF model permits this extension to object-valued distributions supported on a metric space $(\mathcal Y, d)$ under the following assumption.

\begin{assumption}
\label{assmp:mani}
Corresponding to the metric space $(\mathcal Y, d)$, there exists a Riemannian manifold $(M,g)$ such that $\mathcal Y \subseteq M$ and the geodesic distance $d_g$ induced by $g$ satisfies $cd(x,y)\leq d_g(x,y)\leq C d(x,y)$ for some fixed constants $c,C>0$ and all $x,y\in \mathcal Y$.
\end{assumption}

For example, the space of $m\times m$ symmetric positive definite (SPD) matrices forms a Riemannian manifold with the Rao--Fisher metric $d_g(P_1, P_2)=\|\log(P_1^{-1/2}P_2P_1^{-1/2})\|_F$. 
The space of unweighted graphs with $m$ nodes forms a discrete metric space; a natural distance between graphs $O_1$ and $O_2$ is $d(O_1,O_2)=\|\log(\widetilde{L}_1^{-1/2}\widetilde{L}_2\widetilde{L}_1^{-1/2})\|_F$, where $\widetilde{L}_1$ and $\widetilde{L}_2$ are the nearest SPD matrices to the graph Laplacians $L_1$ and $L_2$, respectively. One may use the algorithm of \cite{cheng1998modified} to compute the nearest SPD matrices, or alternatively set $\widetilde{L}_k=L_k+\eta I_m$ for a fixed small $\eta>0$. Another widely used distance for graph-valued data is $\|L_1-L_2\|_F^2$, in which case the data can be embedded in the $m(m+1)/2$-dimensional Euclidean space with its standard topology. We use the Gaussian-BSF model with Riemannian kernel $f^{(n)}_{st}=\zeta(\sigma_n)\exp\left\{-d^2_g(y_s,y_t)/(2\sigma_n^2)\right\}$.

\begin{theorem}[Consistency when oracle distributions are object-valued]
\label{thm:General_Mixture_Consistency}
Let $(V^{0,\infty}_1,V^{0,\infty}_2,\ldots)$ be the oracle clustering for $y^{(\infty)}$.
For each $n$, suppose $y_i\stackrel{\mathrm{indep}}{\sim}G^0_{k}$ if $y_{i}\in V^{0,\infty}_k$.
Set $\rho_n:=(\delta_n\lambda_n\zeta(\sigma_n))^{-1}$ and $f^{(n)}_{st}=\zeta(\sigma_n)\exp\{-d^2_g(y_s,y_t)/(2\sigma_n^2)\}$. Assume that
\begin{enumerate}
\item[(i)] $\mu_k:=\arg\min_z\mathbb{E}_{x\sim G_k^0} d^2(z,x)$ is the unique Fr\'echet mean under $G_k^0$, and \\$D_{\mu,\min}:=\min_{k,\ell\in[K_0],k\not=\ell}d(\mu_k,\mu_\ell)$;
\item[(ii)] Assumptions~\ref{ass1} and \ref{ass2} hold;
\item[(iii)] $\rho_n\succsim(K_{0,n}+1+\iota)^{n}$ for a fixed constant $\iota>0$;
\item[(iv)] $P_{G^0_k}(d(X,\mu_k)>R)\leq\exp\left(-CR^{\nu_n}\right)$ for a fixed constant $C$, any $k\in[K_{0,n}]$ and $R\geq0$, and a sequence $\nu_n$ satisfying $\nu:=\liminf_{n\to\infty}\nu_n>0$;
\item[(v)] $(\log(n))^{2/\nu}/[\sigma_n^2\log(\rho_n)]=o(1)$;
\item[(vi)] $\sigma_n^2\log(\rho_n)/D^2_{\mu,\min}=o(1)$.
\end{enumerate}
Then
\(\mathbb E_{P^{(n)}_0}[\Pi(\mathcal V_n\not\sim(V_1^{0,n},\ldots,V_{K_{0,n}}^{0,n})|y^{(n)})]\stackrel{n\to \infty}{\to} 0.\)
\end{theorem}

Assumption~(i) is a prerequisite for Assumption~(vi). While there may exist multiple minimizers of $\mathbb{E}_{x\sim G_k^0}d^2(z,x)$ in general, the Fr\'echet mean is unique in a wide range of spaces. For instance, in Hadamard spaces, uniqueness is guaranteed \citep{sturm2003probability}; for complete Riemannian manifolds, existence and uniqueness are discussed in detail by \cite{afsari2011riemannian}. Assumptions~(ii), (iii), and (vi) mirror those in Theorem~\ref{thm:GMM_consistency_gen}, while Assumption~(iv) captures the role of the oracle distribution's spread through a tail probability bound. By an argument analogous to the discussion following Theorem~\ref{thm:GMM_consistency_gen}, under mild conditions, if $D_{\mu,\min}/(\log(n))^{1/\nu}\to\infty$, then there exists $(\delta_n\lambda_n,\sigma^2_n)$ yielding clustering consistency.

\begin{remark}
Assumption~(iv) controls the tail behavior of $G_k^0$. When $\nu_n\equiv\nu$ is constant, setting $\nu=2$ makes $G^0_k$ sub-Gaussian and $\nu=1$ makes it sub-exponential. More generally, $\nu$ may be any arbitrarily small positive constant, accommodating distributions with relatively heavy tails while still ensuring clustering consistency.
\end{remark}

In general, the normalizing constant $\zeta(\sigma_n)$ is not tractable for arbitrary $d$, so quantifying $\delta_n\lambda_n$ from $\rho_n$ alone may not be possible. Nevertheless, for the broad class of examples in \cite{said2022gaussian}, $\log(\zeta(\sigma_n))=o(n)$ holds when $\sigma_n\asymp n^{-\beta'}$ with $\beta'>0$. In that case one may take $\delta_n\lambda_n\asymp(K_{0,n}+1+\iota)^{-n}$, since the growth rate of $\log(\delta_n\lambda_n)$ dominates that of $\log(\zeta(\sigma_n))$.

\subsection{Expected misclassification rate with known $K_{0,n}$}
\label{sec:rateknownK}

We present results on the expected misclassification rate under the assumption that $K_{0,n}$ is known and, for simplicity, that the oracle distributions $G^0_k$ are Gaussian. When $K_{0,n}$ is known, the hyperparameters $\delta_n$ and $\lambda_n$ become irrelevant, as their primary role is to control the estimated number of clusters. Mathematically, $\delta_n\lambda_n$ cancels in the ratio $\sfrac{\Pi(\mathcal V_n\not\sim (V^{0,n}_{1},\ldots,V^{0,n}_{K_{0,n}}),|\mathcal V_n|=K_{0,n}\mid y^{(n)})}{\Pi(\mathcal V_n\sim (V^{0,n}_{1},\ldots,V^{0,n}_{K_{0,n}})\mid y^{(n)})}$, and the parameter $\sigma_n^2$ in $f_{st}^{(n)}$ directly governs how partitions are formed.

Corollary~\ref{cor:GMM_consistency_gen} guarantees clustering consistency under appropriate choices of $(\delta_n\lambda_n,\sigma_n^2)$ when $\mathrm{SNR}/\sqrt{p_n\vee\log(n)}\to\infty$. However, selecting such hyperparameters typically requires knowledge of the oracle to avoid both overpartitioning and underpartitioning. Here we establish a stronger result: under the same SNR condition, the expected misclassification rate decays exponentially in $\mathrm{SNR}^2$, provided $\sigma_n^2$ decays sufficiently fast. { We formally state this as Theorem~\ref{thm:mis_class_for_Gaussian} below. A supporting simulation study, confirming that numerical results align with the theorem, is presented in Section~S3.}

\begin{theorem}
\label{thm:mis_class_for_Gaussian}
Suppose $(V^{0,\infty}_1,V^{0,\infty}_2,\ldots)$ is the oracle clustering for $y^{(\infty)}$ and $K_{0,n}$ is known.
Consider the Gaussian-BSF model with $f^{(n)}_{st}=\zeta(\sigma_n)\exp\left\{-d^2_{st}/(2\sigma_n^2)\right\}$.
For each $n$, suppose $y_i\stackrel{\mathrm{indep}}{\sim}\mathrm{N}(\mu_{k},\Sigma_{k})$ if $y_{i}\in V^{0,\infty}_k$.
Set $\Lambda_{\max}:=\max_{k\in[K_{0,n}]}\lambda_{\max}(\Sigma_k)$ and $D_{\mu,{\min}}:=\min_{k,\ell\in[K_{0,n}],k\not=\ell}\|\mu_k-\mu_\ell\|_2$. Suppose Assumption~\ref{ass1} holds and $\mathrm{SNR}/\sqrt{p_n\vee\log(n)}\to\infty$ as $n\to\infty$. Then
\(
\mathbb E_{P^{(n)}_0}[\mathbb{E}(d_H(\mathcal V_n,(V_1^{0,n},\ldots,V_{K_{0,n}}^{0,n}))\mid y^{(n)})]&\precsim\exp(-O(1)\cdot \mathrm{SNR}^2),
\)
for
\(
\sigma_n^2 = o\!\left(\frac{D_{\mu,\min}^2}{n \log(K_{0,n}+1)} \wedge \Lambda_{\max}\right).
\)
\end{theorem}

\begin{remark}
The admissible range for $\sigma_n^2$ depends on the triplet $(K_{0,n},D_{\mu,\min}^2,\Lambda_{\max})$, but the condition requires only that $\sigma_n^2$ be asymptotically smaller than a given upper bound, so one may safely let $\sigma_n^2$ decay with $n$. This is particularly useful in practice when $K_{0,n}$ is known: the risk of oversplitting tightly grouped points is reduced even for small $\sigma_n^2$, while a small $\sigma_n^2$ simultaneously sharpens sensitivity to inter-cluster separation. Thus, when $K_{0,n}$ is known, a small $\sigma_n^2$ is both theoretically justified and practically robust.
\end{remark}

\section{Proof sketch and useful techniques}

{
The results in this section serve as the primary technical tools for the proof of the main result, which is Lemma~\ref{lem:subset}. Section~\ref{sec:refinement} introduces a notion of clustering refinement that allows posterior probability ratios to be decomposed through local refinements of partitions. Section~\ref{sec:bounding_ratios} then provides the determinant bounds needed to control these ratios. Among the results there, Lemmas~\ref{lem:ratio_det1} and \ref{lem:ratio_det2} play the central role in the proof of Lemma~\ref{lem:subset}: together with the refinement decomposition of Section~\ref{sec:refinement}, they furnish an upper bound on the posterior probability ratio for any two candidate partitions. The remaining lemmas in Section~\ref{sec:bounding_ratios} collect standard linear algebra tools used to establish these two key lemmas. The techniques and bounds developed here may also be of independent interest beyond the scope of consistency theory. Throughout, we suppress super- and subscripts of $n$ where no ambiguity arises.
}

\subsection{Refinement}
\label{sec:refinement}

An essential step in the proof of Lemma~\ref{lem:subset} is to evaluate a bound for
$\frac{\Pi^*(\mathcal V_n\not=(V^0_{1},\ldots,V^0_{K_0})|y^{(n)})}{\Pi^*(\mathcal V_n=(V^0_{1},\ldots,V^0_{K_0})|y^{(n)})}$,
which involves the ratio
$\sfrac{\Pi^*(\mathcal V_n=(V_1,\ldots, V_K)|y^{(n)})}{\Pi^*(\mathcal V_n=(V^0_1,\ldots, V^0_{K_0})|y^{(n)})}$
for an arbitrary partition $(V_1,\ldots,V_K)$ differt from the truth $(V^0_1,\ldots, V^0_{K_0})$, with the true partition in the denominator.
{ To obtain suitable bounds for this ratio, we devise a strategy based on a {\em refinement} of an existing clustering, defined below.}

\begin{definition}[Refinement]
Consider two partitions of $n$ nodes, $\mathcal V^1_n=(V^1_1,\ldots,V^1_{K_1})$ and $\mathcal V^2_n=(V^2_1,\ldots,V^2_{K_2})$. We say $\mathcal V^1_n$ is a refinement of $\mathcal V^2_n$ if for any $i\in[K_1]$, there exists $j\in[K_2]$ such that $V^1_i\subset V^2_j$.
\end{definition}

Intuitively, a refinement $\mathcal V^1_n$ of $\mathcal V^2_n$ is obtained by splitting some clusters of $\mathcal V^2_n$. As two elementary examples: any partition of $n$ nodes is a refinement of the single-cluster partition, and the partition into $n$ singletons is a refinement of every partition. { Figure~\ref{fig:refinement}(a) gives a concrete illustration of refinement.}

The focus on refinements is motivated by three considerations. First, working with ratios of posterior probabilities is more convenient than working with individual probabilities, since the normalizing constant cancels.
Second, for two partitions $(V_1,\ldots,V_K)$ and $(V^0_1,\ldots,V^0_{K_0})$ of $n$ nodes, define $W_{ij}=V_i\cap V^0_j$ for all $i\in[K]$ and $j\in[K_0]$; when $W_{ij}=\emptyset$, set $|L_{W_{ij}}+|W_{ij}|^{-1}J|\stackrel{\Delta}{=}1$ by convention. { Figure~\ref{fig:refinement}(b) illustrates this construction.} From \eqref{eq:post_prob_V_n} it follows that
\(
\begin{aligned}
&\frac{\Pi^*(\mathcal V_n=(V_1,\ldots, V_K)|y^{(n)})}{\Pi^*(\mathcal V_n=(V^0_1,\ldots, V^0_{K_0})|y^{(n)})}
=(\delta\lambda)^{K-K_0}\cdot\frac{\prod_{i=1}^K|L_{V_i}+|V_i|^{-1}J|}{\prod_{j=1}^{K_0}|L_{V^0_j}+|V_j^0|^{-1}J|}\\
=&(\delta\lambda)^{K-K_0}\cdot\left(\prod_{i=1}^K\frac{|L_{V_i}+|V_i|^{-1}J|}{\prod_{j=1}^{K_0}|L_{W_{ij}}+|W_{ij}|^{-1}J|}\right)\cdot\left(\prod_{j=1}^{K_0}\frac{\prod_{i=1}^{K}|L_{W_{ij}}+|W_{ij}|^{-1}J|}{|L_{V_{j}^0}+|V_j^0|^{-1}J|}\right),
\end{aligned}
\)
where the last expression links naturally to refinement. For any $i\in[K]$, the nonempty sets among $\{W_{ij}\}_{j\in[K_0]}$ form a refinement of $V_i$, contributing the factor $\sfrac{|L_{V_i}+|V_i|^{-1}J|}{\prod_{j=1}^{K_0}|L_{W_{ij}}+|W_{ij}|^{-1}J|}$ in the first product. An analogous statement holds for each $j\in[K_0]$. Hence, locally each partition can be viewed as a refinement of the other, and it suffices to bound all determinant ratios through refinement. Third, the refinement argument automatically handles relabeling: if two partitions are equal up to relabeling, the collection of nonempty $W_{ij}$ sets coincides with either partition.
{
This refinement decomposition is applied at the outset of the proof of Lemma~\ref{lem:subset} to obtain display~(S2.1) of the Supplementary Material.
}

\begin{figure}[ht]
\centering
\begin{tikzpicture}[
    nd/.style={circle, draw=black, fill=white, thin,
               minimum size=7mm, font=\small},
    frame/.style={rounded corners=10pt, draw=black, thick, fill=none},
    >=Stealth
]

\begin{scope}[xshift=0cm, yshift=0cm]

  \begin{scope}[xshift=0cm]
    \node[nd] (cp) at (0, 0)    {$p$};
    \node[nd] (cq) at (0,-1.3)  {$q$};
    \node[nd] (cr) at (0,-2.6)  {$r$};
    \node[nd] (cs) at (0,-3.9)  {$s$};

    \draw[frame] (-0.48, 0.48) rectangle (0.48,-3.08);
    \draw[frame] (-0.48,-3.42) rectangle (0.48,-4.38);

    \node[font=\scriptsize, above] at (0, 0.48)    {$V_1$};
    \node[font=\scriptsize, below] at (0, -4.38)   {$V_2$};

    \node[font=\small\bfseries] at (0, 1.2) {$\mathcal{V}=(V_1, V_2)$};
  \end{scope}

  \draw[->, thick, black!70, line width=1.2pt] (0.85,-1.95) -- (1.85,-1.95);
  \node[font=\scriptsize, align=center, black!60] at (1.35, -1.35)
    {split\\$V_1$};

  \begin{scope}[xshift=2.7cm]
    \node[nd] (fp) at (0, 0)    {$p$};
    \node[nd] (fq) at (0,-1.3)  {$q$};
    \node[nd] (fr) at (0,-2.6)  {$r$};
    \node[nd] (fs) at (0,-3.9)  {$s$};

    \draw[frame] (-0.48, 0.48) rectangle (0.48,-1.78);
    \draw[frame] (-0.48,-2.12) rectangle (0.48,-3.08);
    \draw[frame] (-0.48,-3.42) rectangle (0.48,-4.38);

    \node[font=\scriptsize, above] at (0,  0.48)  {$V'_{1}$};
    \node[font=\scriptsize, right] at (0.48,-2.6) {$V'_{2}$};
    \node[font=\scriptsize, below] at (0, -4.38)  {$V'_{3}$};

    \node[font=\small\bfseries] at (0, 1.2) {$\mathcal{V}'=(V'_{1}, V'_{2}, V'_{3})$};
  \end{scope}

  \node[font=\small, align=center] at (1.35, -6.15)
    {(a) $\mathcal{V}'$ is a refinement of $\mathcal{V}$:\\
     \footnotesize $V_1$ is split into $V'_{1}$ and $V'_{2}$};

\end{scope}

\begin{scope}[xshift=5.8cm, yshift=0cm]

  \def\cw{2.7}
  \def\ch{2.05}
  \def\hdr{0.8}
  \def\lhdr{1.1}

  \draw[thick, black]
    (\lhdr, -0.05) rectangle (\lhdr+\cw, -\hdr-0.05);
  \node[font=\small\bfseries] at (\lhdr+\cw/2, -\hdr/2-0.05)
    {$V^0_1=\{1,2,4\}$};

  \draw[thick, black]
    (\lhdr+\cw, -0.05) rectangle (\lhdr+2*\cw, -\hdr-0.05);
  \node[font=\small\bfseries] at (\lhdr+1.5*\cw, -\hdr/2-0.05)
    {$V^0_2=\{3,5\}$};

  \draw[thick, black]
    (0, -\hdr-0.05) rectangle (\lhdr, -\hdr-\ch-0.05);
  \node[font=\small\bfseries] at (\lhdr/2, -\hdr-\ch/2-0.05+0.22) {$V_1$};
  \node[font=\scriptsize]     at (\lhdr/2, -\hdr-\ch/2-0.05-0.22) {$\{1,2,3\}$};

  \draw[thick, black]
    (0, -\hdr-\ch-0.05) rectangle (\lhdr, -\hdr-2*\ch-0.05);
  \node[font=\small\bfseries] at (\lhdr/2, -\hdr-\ch-\ch/2-0.05+0.22) {$V_2$};
  \node[font=\scriptsize]     at (\lhdr/2, -\hdr-\ch-\ch/2-0.05-0.22) {$\{4,5\}$};

  \draw[thick, black]
    (\lhdr, -\hdr-0.05) rectangle (\lhdr+2*\cw, -\hdr-2*\ch-0.05);
  \draw[thick, black]
    (\lhdr+\cw, -\hdr-0.05) -- (\lhdr+\cw, -\hdr-2*\ch-0.05);
  \draw[thick, black]
    (\lhdr, -\hdr-\ch-0.05) -- (\lhdr+2*\cw, -\hdr-\ch-0.05);

  \node[font=\normalsize\bfseries] at (\lhdr+\cw/2,   -\hdr-\ch/2-0.05+0.25)      {$W_{11}$};
  \node[font=\small]               at (\lhdr+\cw/2,   -\hdr-\ch/2-0.05-0.25)      {$\{1,2\}$};

  \node[font=\normalsize\bfseries] at (\lhdr+1.5*\cw, -\hdr-\ch/2-0.05+0.25)      {$W_{12}$};
  \node[font=\small]               at (\lhdr+1.5*\cw, -\hdr-\ch/2-0.05-0.25)      {$\{3\}$};

  \node[font=\normalsize\bfseries] at (\lhdr+\cw/2,   -\hdr-\ch-\ch/2-0.05+0.25)  {$W_{21}$};
  \node[font=\small]               at (\lhdr+\cw/2,   -\hdr-\ch-\ch/2-0.05-0.25)  {$\{4\}$};

  \node[font=\normalsize\bfseries] at (\lhdr+1.5*\cw, -\hdr-\ch-\ch/2-0.05+0.25)  {$W_{22}$};
  \node[font=\small]               at (\lhdr+1.5*\cw, -\hdr-\ch-\ch/2-0.05-0.25)  {$\{5\}$};

  \node[font=\small, align=center]
    at (\lhdr+\cw, 0.85)
    {$\mathcal{V}_n=(V_1,V_2)$ vs.\ $\mathcal{V}^0_n=(V^0_1,V^0_2)$};

  \node[font=\small, align=center]
    at (\lhdr+\cw-0.5, -\hdr-2*\ch-1.2)
    {(b) Intersection sets $W_{ij} = V_i \cap V^0_j$};

\end{scope}

\end{tikzpicture}
\caption{
  Refinement and the $W_{ij}$ intersection construction.
  (a)~Partition $\mathcal{V}'$ is a refinement of $\mathcal{V}$:
  cluster $V_1$ is split into $V'_{1}$ and $V'_{2}$, while $V_2$ is unchanged ($V'_3=V_2$).
  (b)~For two partitions $\mathcal{V}_n=(V_1,V_2)$ and
  $\mathcal{V}^0_n=(V^0_1,V^0_2)$ of $n=5$ nodes, the intersection sets
  $W_{ij}=V_i\cap V^0_j$ arrange into a $2\times 2$ grid.
  Each row $\{W_{i1},W_{i2}\}$ partitions $V_i$ (a refinement of $V_i$);
  each column $\{W_{1j},W_{2j}\}$ partitions $V^0_j$ (a refinement of $V^0_j$).
}
\label{fig:refinement}
\end{figure}

\subsection{Bounding the ratio of determinants}
\label{sec:bounding_ratios}

Let $(V_1,\ldots,V_K)$ be any partition of $n$ nodes with $|V_i|=n_i$, and let $(V_1^0)$ denote the single-cluster partition. We seek bounds on the ratio of the associated determinants: $\prod_{i=1}^K|L_{V_i}+n_i^{-1}J|$ and $|L_n+n^{-1}J|$.
{
We first collect the linear algebra tools needed to derive these bounds.}

\begin{lemma}
\label{lem:eigens}
Suppose $L_n$ has eigenvalues $\lambda_n\geq\lambda_{n-1}\geq\cdots\geq\lambda_1=0$. Then for any $a,b\in\mathbb R$, the eigenvalues of $L_n+aI+bJ$ are $\lambda_n+a,\ldots,\lambda_2+a,nb+a$.
\end{lemma}

\begin{proof}
The proof is a direct application of Lemma~4.5 in \cite{bapat14}.
\end{proof}

\begin{remark}
\label{remark:det_Ln+J/n}
A direct consequence of Lemma~\ref{lem:eigens} is that $|L_n+aI+bJ|=(nb+a)\prod_{i=2}^n(\lambda_i+a)$. In the special case $a=0$, $b=n^{-1}$:
\[
\label{eq:det_Ln+1/nJ}
|L_n+n^{-1}J|=\prod_{i=2}^n\lambda_i=n|L_n[1]|,
\]
where the last equality follows from Kirchhoff's matrix theorem.
\end{remark}

\begin{lemma}[Matrix Determinant Lemma]
\label{lem:matrix_det_lamma}
Let $M$ be an invertible $n\times n$ matrix and let $\mathbf{a}$ and $\mathbf{b}$ be column vectors in $\mathbb R^n$. Then $|M+\mathbf{a}\mathbf{b}^T|=|M|(1+\mathbf{b}^T M^{-1}\mathbf{a})$.
\end{lemma}

\begin{lemma}
\label{lem:ineq_det}
Suppose $A$ is a symmetric positive definite $n\times n$ matrix and $B$ is a symmetric nonnegative definite $n\times n$ matrix. Then $|A+B|\geq|A|$.
\end{lemma}

The proofs of Lemmas~\ref{lem:matrix_det_lamma} and \ref{lem:ineq_det} follow from Theorems~18.1.1 and 18.1.6 of \cite{david1997}, respectively; see also \cite{steven2019} and Section~0.8.5 of \cite{Horn_Johnson_2012}.

{
The following lemma provides an upper bound for $\sfrac{|L_n + n^{-1}J|}{\prod_{i=1}^K |L_{V_i} + n_i^{-1}J|}$, corresponding to the case where the denominator partition is a refinement of the numerator partition. This bound is invoked at display~(S2.3) of the Supplementary Material in the proof of Lemma~\ref{lem:subset}.}

\begin{lemma}
\label{lem:ratio_det1}
Suppose there exists $\varepsilon>0$ such that $f^{(n)}_{st}\leq\varepsilon\leq f^{(n)}_{s't'}$ for any $s\not\sim t$ and $s'\sim t'$ under $\mathcal V_n=(V_1,\ldots,V_K)$. Then
\[
\frac{|L_n+n^{-1}J|}{\prod_{i=1}^K|L_{V_i}+n_i^{-1}J|}\leq (n\varepsilon)^{K-1}\prod_{i=1}^K\prod_{j=2}^{n_{i}}\left(1+\frac{(n-n_i)\varepsilon}{\lambda_{ij}}\right),
\]
where $\prod_{j=2}^1x_j\stackrel{\Delta}{=}1$, $n_i=|V_i|$, and $\lambda_{ij}$ is the $j$-th smallest eigenvalue of $L_{V_i}$.
\end{lemma}

{
The key step in the proof of Lemma~\ref{lem:ratio_det1} is to replace the cross-cluster edge weights $f_{st}^{(n)}$ in the numerator Laplacian by their upper bound $\varepsilon$. This dominates the Laplacian in the positive semidefinite order and yields a block structure consistent with the partition $(V_1,\ldots,V_K)$, from which the determinant bound follows.
}

\begin{proof}[Proof of Lemma~\ref{lem:ratio_det1}]
Consider the Laplacian $A$ on $n$ nodes with edge weights $a_{st}=f^{(n)}_{st}1_{(s\sim t)}+\varepsilon 1_{(s\not\sim t)}\geq0$, and the Laplacian $B$ on $n$ nodes with edge weights $b_{st}=(\varepsilon-f^{(n)}_{st})1_{(s\not\sim t)}\geq0$ for $s,t\in[n]$.

Since $a_{st}-b_{st}=f^{(n)}_{st}$ for all $s,t\in[n]$, we have $L_n=A-B$. Hence,
\[
\label{ineq:det_L_n}
|L_n+n^{-1}J|=|A-B+n^{-1}J|\leq|A+n^{-1}J|,
\]
where the inequality holds because $L_n+n^{-1}J$ is positive definite and $B$ is nonnegative definite.
To compute $|A+n^{-1}J|$, apply Lemma~\ref{lem:matrix_det_lamma}:
\[
\label{eq:det_A_n}
|A+n^{-1}J|=|(A+\varepsilon J)+(n^{-1}-\varepsilon)J|=|A+\varepsilon J|\cdot(1-(\varepsilon-n^{-1})\mathbf{1}_n^T(A+\varepsilon J)^{-1}\mathbf{1}_n).
\]
Note that the determinant $|A+\varepsilon J|$ equals the determinant of a block-diagonal matrix:
\(
\left|\begin{pmatrix}
L_{V_1}+(n-n_1)\varepsilon I+\varepsilon J & & & \\
 & L_{V_2}+(n-n_2)\varepsilon I+\varepsilon J & & \\
 & & \ddots & \\
 & & & L_{V_K}+(n-n_K)\varepsilon I+\varepsilon J
\end{pmatrix}\right|.
\)
By Lemma~\ref{lem:eigens},
\[
\label{eq:det_A_epsilon}
\begin{aligned}
|A+\varepsilon J|&=\prod_{i=1}^K|L_{V_i}+(n-n_i)\varepsilon I+\varepsilon J|
=\prod_{i=1}^{K}(n\varepsilon)\prod_{j=2}^{n_i}(\lambda_{ij}+(n-n_i)\varepsilon)\\
&=(n\varepsilon)^K\prod_{i=1}^K\prod_{j=2}^{n_i}(\lambda_{ij}+(n-n_i)\varepsilon),
\end{aligned}
\]
and $(A+\varepsilon J)\mathbf{1}_n=n\varepsilon\mathbf{1}_n$, so $(A+\varepsilon J)^{-1}\mathbf{1}_n=(n\varepsilon)^{-1}\mathbf{1}_n$. Hence,
\[
\label{eq:det_A_epsilon2}
\mathbf{1}_n^T(A+\varepsilon J)^{-1}\mathbf{1}_n=\frac{1}{n\varepsilon}\mathbf{1}_n^T\mathbf{1}_n=\frac{1}{\varepsilon}.
\]
Combining \eqref{ineq:det_L_n}, \eqref{eq:det_A_n}, \eqref{eq:det_A_epsilon}, and \eqref{eq:det_A_epsilon2} gives
\(
|L_n+n^{-1}J|\leq (n\varepsilon)^{K-1}\prod_{i=1}^K\prod_{j=2}^{n_i}(\lambda_{ij}+(n-n_i)\varepsilon).
\)
By Lemma~\ref{lem:eigens}, $|L_{V_i}+n_i^{-1}J|=\prod_{j=2}^{n_i}\lambda_{ij}$, and therefore
\(
\frac{|L_n+n^{-1}J|}{\prod_{i=1}^K|L_{V_i}+n_i^{-1}J|}
\leq \frac{(n\varepsilon)^{K-1}\prod_{i=1}^K\prod_{j=2}^{n_i}(\lambda_{ij}+(n-n_i)\varepsilon)}{\prod_{i=1}^K\prod_{j=2}^{n_i}\lambda_{ij}}
=(n\varepsilon)^{K-1}\prod_{i=1}^K\prod_{j=2}^{n_i}\left(1+\frac{(n-n_i)\varepsilon}{\lambda_{ij}}\right).
\)
\end{proof}

{
Next we derive an upper bound for $\sfrac{\prod_{i=1}^K|L_{V_i}+n_i^{-1}J|}{|L_n+n^{-1}J|}$, the complementary regime in which the numerator partition is a refinement of the denominator. Lemma~\ref{lem:ineq_L_n} below reduces this to a lower bound on $|T_K[1]|$, the cofactor of a reduced Laplacian constructed from aggregated cluster weights, thereby converting a ratio of determinants into the evaluation of a single Laplacian determinant. Lemma~\ref{lem:ratio_det2} then delivers the bound and is invoked at display~(S2.6) of the Supplementary Material in the proof of Lemma~\ref{lem:subset}.
}

\begin{lemma}
\label{lem:ineq_L_n}
Let $\tau_{st}:=\sum_{i\in V_s}\sum_{j\in V_t}f_{ij}$ for $s,t\in[K]$.
Let $T_K$ be the Laplacian on $K$ nodes with edge weights $\tau_{st}$, i.e.,
\[
T_K=\begin{pmatrix}
\sum_{t\not=1}\tau_{1t}&-\tau_{12}&\cdots&-\tau_{1K}\\
-\tau_{12} & \sum_{t\not=2}\tau_{2t}&\cdots&-\tau_{2K}\\
\vdots&\vdots&\ddots&\vdots\\
-\tau_{1K}&\cdots&\cdots&\sum_{t\not=K}\tau_{KK}
\end{pmatrix}.
\]
Then
\[
|L_n[1]|\geq |T_K[1]|\cdot\prod_{i=1}^K|L_{V_i}[1]|.
\]
\end{lemma}

\begin{proof}[Proof of Lemma~\ref{lem:ineq_L_n}]
The proof relies on a carefully constructed application of Kirchhoff's matrix theorem. Let $f_{ij}$ denote the edge weights. Then
\(
|L_n[1]|=\sum_{T_n\in\mathcal T_n}\prod_{(i,j)\in T_n}f_{ij},
\)
where $\mathcal T_n$ is the set of all spanning trees on $n$ nodes.
For the partition $\mathcal V_n=(V_1,\ldots,V_K)$, let $\mathcal{E}_k$ denote the set of spanning trees on $V_k$.
Define a restricted set $\mathcal T'_n\subset\mathcal T_n$ as the product $(\times_{k=1}^K\mathcal{E}_k)\times\mathcal{C}$, where $\mathcal{C}$ is the set of spanning trees with $K-1$ edges connecting exactly one node from each $V_k$, and $\times$ denotes the external direct product. Elements of $\mathcal T'_n$ are spanning trees on all $n$ nodes obtained by concatenating one tree from each $\mathcal{E}_k$ using a connecting tree from $\mathcal{C}$.

A typical element $T_n\in\mathcal T'_n$ decomposes as $\{T_{V_1},\ldots,T_{V_K},C\}$, where $T_{V_k}\in\mathcal{E}_k$ and $C\in\mathcal{C}$, with $\prod_{(i,j)\in T_n}f_{ij}=\bigl(\prod_k\prod_{(i,j)\in T_{V_k}}f_{ij}\bigr)\bigl(\prod_{(i,j)\in C}f_{ij}\bigr)$.
Hence $|L_n[1]|\geq\sum_{T_n\in\mathcal T'_n}\prod_{(i,j)\in T_n}f_{ij}$, and since $\mathcal T'_n$ is a product space,
\(
\sum_{T_n\in\mathcal T'_n}\prod_{(i,j)\in T_n}f_{ij}
=\left(\prod_k\sum_{T_{V_k}\in\mathcal{E}_k}\prod_{(i,j)\in T_{V_k}}f_{ij}\right)\left(\sum_{C\in\mathcal{C}}\prod_{(i,j)\in C}f_{ij}\right).
\)
The second factor equals $|T_K[1]|$, since $\sum_{C\in\mathcal{C}}\prod_{(i,j)\in C}f_{ij}=\sum_{(s,t)\in[K]}\tau_{st}$ with $\tau_{st}=\sum_{i\in V_s}\sum_{j\in V_t}f_{ij}$. The first factor equals $\prod_k|L_{V_k}[1]|$ by Kirchhoff's theorem. This completes the proof.
\end{proof}

\begin{lemma}
\label{lem:ratio_det2}
Suppose there exists $\gamma>0$ such that $f^{(n)}_{st}\geq\gamma$ for all $s,t\in[n]$. Then
\(
\sup_{\mathcal V_n:|\mathcal V_n|=K}\frac{\prod_{i=1}^K|L_{V_i}+n_i^{-1}J|}{|L_n+n^{-1}J|}\leq \left(\frac{1}{n\gamma}\right)^{K-1}.
\)
\end{lemma}

{
The proof of Lemma~\ref{lem:ratio_det2} builds on Lemma~\ref{lem:ineq_L_n}. It remains to lower-bound $|T_K[1]|$. Under $f^{(n)}_{st}\geq\gamma$, each aggregated weight satisfies $\tau_{st}\geq n_sn_t\gamma$, allowing $T_K$ to be compared with the Laplacian whose edge weights are $n_sn_t\gamma$ and thereby yielding the required lower bound.
}

\begin{proof}[Proof of Lemma~\ref{lem:ratio_det2}]
Let $T_K^\gamma$ be the Laplacian on $K$ nodes with edge weights $\tau_{st}^\gamma=n_sn_t\gamma$ for $s,t\in[K]$. Then $T_K-T_K^\gamma$ is a Laplacian with edge weights $\tau^\Delta_{st}=\sum_{i\in V_s}\sum_{j\in V_t}f^{(n)}_{st}-n_sn_t\gamma\geq0$, so $T_K^\gamma[n]$ is positive definite and $(T_K-T_K^\gamma)[n]$ is nonnegative definite. By Lemma~\ref{lem:ineq_det},
\[
\label{ineq:T_K}
\begin{aligned}
|T_K[1]|&=|T_K[n]|\geq|T_K^\gamma[n]|\\
&=\gamma^{K-1}\left|\begin{matrix}
n_1(n-n_1)&-n_1n_2&\cdots&-n_1n_{K-1}\\
-n_1n_2&n_2(n-n_2)&\cdots&-n_2n_{K-1}\\
\vdots&\vdots&\ddots&\vdots\\
-n_1n_{K-1}&-n_2n_{K-1}&\cdots&n_{K-1}(n-n_{K-1})
\end{matrix}\right|\\
&=\gamma^{K-1}(-1)^{K-1}\prod_{i=1}^{K-1}(-n_in)\cdot\left(1+\sum_{i=1}^{K-1}\frac{n_i^2}{-n_in}\right)\\
&=\gamma^{K-1}n^{K-2}\prod_{i=1}^Kn_i.
\end{aligned}
\]
Hence, for any $K\in\mathbb{N}^+$ and partition $\mathcal V_n=(V_1,\ldots,V_K)$,
\(
\frac{\prod_{i=1}^K|L_{V_i}+n_i^{-1}J|}{|L_n+n^{-1}J|}
\stackrel{(a)}{=}\frac{\prod_{i=1}^K(n_i|L_{V_i}[1]|)}{n|L_n[1]|}
\stackrel{(b)}{\leq}\frac{\prod_{i=1}^Kn_i}{n}\cdot\frac{1}{|T_K[1]|}
\stackrel{(c)}{\leq}\left(\frac{1}{n\gamma}\right)^{K-1},
\)
where $(a)$, $(b)$, and $(c)$ follow from \eqref{eq:det_Ln+1/nJ}, Lemma~\ref{lem:ineq_L_n}, and \eqref{ineq:T_K}, respectively.
\end{proof}

\section{Discussion}
While most of our theoretical results address the case where the number of clusters is unknown, we also analyze a simplified scenario with known number of clusters, deriving an upper bound for the misclassification rate. We demonstrate the practical utility of our approach through illustrative examples and a small simulation in the Supplementary. Several promising directions remain for future investigation.
First, we primarily focus on recovering oracle clustering, when data within the same cluster are i.i.d. from some component distribution. The restriction could be relaxed to dependent data, since most of the theoretical developments in Section~\ref{section:main} do not need the i.i.d. condition. 
Second, for the consistency theory, we focus on the case where oracle clustering is asymptotically identifiable via $n$-dependent separation conditions. One could extend to the case when the oracle clustering is only partially identifiable, subject to a Bayes misclustering error. However, intuitively, reaching the Bayes error would require stronger assumptions than the ones used in this article, but the empirical results in the Supplementary are promising in this regard. 
Third, for consistency on graphical model-based clustering models, we chose to study the spanning forest graph due to its good empirical performance and mathematical tractability. Clearly, there is a large family of graphs one could consider, including the graphs whose edge formation may be influenced by external covariates. It would be interesting to expand the theory in such a new class of model-based clustering methods.
Our illustration of the Gaussian oracle model allows the dimension to grow moderately with the sample size, without requiring additional assumptions. Future research will explore the standard high-dimensional setting and investigate the structural properties necessary for consistency guarantees of the BSF model.
{ Another direction will be to put hyperpriors on the prior parameters and establish clustering consistency results through an analogous degeneracy characterization as in Proposition 2 of \cite{Ascolani_2022}.}

\begin{supplement}
\stitle{Supplement to ``Consistency of Graphical Model-based Clustering: Robust Clustering using Bayesian Spanning Forest''}
\sdescription{This file contains (1) an illustration of the matrix-tree theorem on a graph of three nodes, (2) the proofs of Lemma 3.3, Lemma 3.5, Lemma 3.6, Theorem 3.7, Theorem 4.2, Theorem 4.4, and Theorem 4.5 in the main paper, (3) simulation study for investigating the finite-sample behavior of the BSF model.}
\end{supplement}

\section*{Disclosure of the use of generative AI}
In preparing this paper, the authors used
Claude Code to assist with code development and simulation. The authors
take full responsibility for all content presented.

\section*{Acknowledgement}
We would like to thank the Editor, Associate Editor, and two anonymous referees for their helpful comments that substantially have helped to improve the quality of the manuscript.

\numberwithin{equation}{section}
\renewcommand{\thesection}{S\arabic{section}}
\renewcommand{\theequation}{S\arabic{section}.\arabic{equation}}
\renewcommand{\thetheorem}{S\arabic{section}.\arabic{theorem}}
\renewcommand{\thelemma}{S\arabic{section}.\arabic{lemma}}
\renewcommand{\thecorollary}{S\arabic{section}.\arabic{corollary}}

\section{Illustration of the matrix-tree theorem}

{In Figure~\ref{fig:matrix_tree}, we illustrate the directed Matrix-Tree Theorem for a cluster of 3 nodes, i.e., $V_k = \{1, 2, 3\}$ with directed edge weights $w_{ij}$.
Considering root $r=1$ (red), there are exactly three directed spanning trees rooted at $1$, each with edges oriented away from the root. We have $\sum_{T\in\mathcal{T}_1}\prod_{(i,j)\in T} w_{ij}\;=\; w_{12}w_{13} + w_{12}w_{23} + w_{13}w_{32}\;=\; \big|L_{V_k}[1]\big|$ where 
\(
\quad L_{V_k} =
    \begin{pmatrix}
      w_{12}+w_{13} & -w_{12} & -w_{13} \\
      -w_{21} & w_{21}+w_{23} & -w_{23} \\
      -w_{31} & -w_{32} & w_{31}+w_{32}
    \end{pmatrix}.
\)
}

\begin{figure}[ht]
\centering
\resizebox{\linewidth}{!}{
\begin{tikzpicture}[
    scale=0.99,
    tnode/.style={circle, draw=nodecolor, fill=white, thick,
                  minimum size=6mm, font=\scriptsize},
    rnode/.style={circle, draw=rootcolor, fill=white, very thick,
                  minimum size=6mm, font=\scriptsize},
    >=Stealth,
    lbl/.style={font=\scriptsize, midway, fill=white, inner sep=1pt}
]

\begin{scope}[xshift=0.2cm, yshift=0cm]
  \node[rnode] (g1) at ( 0.0,  2.1) {$y_1$};
  \node[tnode] (g2) at (-1.3,  0.0) {$y_2$};
  \node[tnode] (g3) at ( 1.3,  0.0) {$y_3$};

  \draw[->, thick, edgecolor, bend right=18] (g1) to (g2);
  \draw[->, thick, edgecolor, bend right=18] (g2) to (g1);
  \draw[->, thick, edgecolor, bend left=18] (g1) to (g3);
  \draw[->, thick, edgecolor, bend left=18] (g3) to (g1);
  \draw[->, thick, edgecolor, bend left=18] (g2) to (g3);
  \draw[->, thick, edgecolor, bend left=18] (g3) to (g2);

\end{scope}

\draw[->, very thick, gray!60, line width=1.4pt] (1.9, 1.05) -- (2.5, 1.05);

\begin{scope}[xshift=3.7cm, yshift=0cm]
  \node[rnode] (a1) at ( 0.0,  1.7) {$y_1$};
  \node[tnode] (a2) at (-0.95,-0.1) {$y_2$};
  \node[tnode] (a3) at ( 0.95,-0.1) {$y_3$};
  \draw[->, thick, edgecolor] (a1) -- node[lbl,left=1pt]  {$w_{12}$} (a2);
  \draw[->, thick, edgecolor] (a1) -- node[lbl,right=1pt] {$w_{13}$} (a3);
\end{scope}

\node[font=\large] at (5.15, 0.8) {$+$};

\begin{scope}[xshift=6.8cm, yshift=0cm]
  \node[rnode] (b1) at ( 0.0,  1.7) {$y_1$};
  \node[tnode] (b2) at (-0.95,-0.1) {$y_2$};
  \node[tnode] (b3) at ( 0.95,-0.1) {$y_3$};
  \draw[->, thick, edgecolor] (b1) -- node[lbl,left=1pt]  {$w_{12}$} (b2);
  \draw[->, thick, edgecolor] (b2) -- node[lbl,below=1pt] {$w_{23}$} (b3);
\end{scope}

\node[font=\large] at (8.25, 0.8) {$+$};

\begin{scope}[xshift=9.9cm, yshift=0cm]
  \node[rnode] (c1) at ( 0.0,  1.7) {$y_1$};
  \node[tnode] (c2) at (-0.95,-0.1) {$y_2$};
  \node[tnode] (c3) at ( 0.95,-0.1) {$y_3$};
  \draw[->, thick, edgecolor] (c1) -- node[lbl,right=1pt] {$w_{13}$} (c3);
  \draw[->, thick, edgecolor] (c3) -- node[lbl,below=1pt] {$w_{32}$} (c2);
\end{scope}

\end{tikzpicture}
}
\caption{
  The directed matrix-tree theorem for a cluster of 3 nodes.
}
\label{fig:matrix_tree}
\end{figure}

\section{Proofs of results and specific examples in the main paper}

\begin{proof}[Proof of Lemma~3.3]

Let $\mathcal V^{K,n}:=\{(V^n_1,\ldots,V^n_K):\bigcup_{k=1}^KV^n_k=[n],V^n_i\bigcap V^n_j=\emptyset,\\\forall i\not=j\}$ be the set of all possible unordered partitions of $y^{(n)}$ into $K$ clusters. For notational convenience, let $n_{0,i}:=|V^{0,n}_i|$ and $n_i:=|V^n_i|$. Let $W^n_{ij}=V^n_i\bigcap V^n_{0,j}$ with $|W^n_{ij}|=m_{ij}$. For any $i\in[K]$, let $a_i$ denote the number of non-empty $W^n_{ij}(j\in[K_0])$; similarly, for any $j\in[K_0]$, let $b_j$ denote the number of non-empty $W^n_{ij}(i\in[K])$. For any non-empty $W^n_{ij}$, we form a Laplacian and we denote its eigenvalues by $\{\lambda_{ijk}\}_{k=1}^{m_{ij}}$. Let $K^*$ be the number of non-empty $W^n_{ij}$'s. Clearly, we have $K^*=\sum_{i=1}^Ka_i=\sum_{j=1}^{K_0}b_j$. In the following, for simplicity, we drop the superscript $n$ in $V_k^{0,n}$, $V_k^{n}$, and $y^{(n)}$ and the subscript $n$ in $K_{0,n}$, $\delta_n$ and $\lambda_n$. Set $\varepsilon:=\sup_{s\not\sim t;s,t\in[n]}f^{(n)}_{st}$ and $\gamma:=\inf_{s'\sim t';s',t'\in[n]}f^{(n)}_{s't'}$ where equivalence relation ``$\sim$" is under the oracle clustering. Then we can write
\[\label{ineq:upperbound}
\begin{aligned}
	&\frac{\Pi^*(\mathcal V_n=(V_1,\ldots,V_K)|y)}{\Pi^*(\mathcal V_n=(V_{1}^0,\ldots,V^0_{K_0})|y)}\\
 =&\frac{C_n\prod_{k=1}^K\left[\lambda\cdot\left(\sum_{\text{all }E_k}\prod_{(i,j)\in E_k}f(y_i|y_j)\right)\cdot\left(\sum_{i\in V_k}r(y_i|y_0)\right)\right]}{C_n\prod_{k=1}^{K_0}\left[\lambda\cdot\left(\sum_{\text{all }E_k}\prod_{(i,j)\in E_k}f(y_i|y_j)\right)\cdot\left(\sum_{i\in V^{0}_k}r(y_i|y_0)\right)\right]}\\
	\leq& \frac{\prod_{i=1}^{K}\left[\lambda\cdot\left|L_{V_i}[1]\right|\cdot n_{i}C_2\delta\right]}{\prod_{j=1}^{K_0}\left[\lambda\cdot\left|L_{V^{0}_j}[1]\right|\cdot n_{0,j}C_1\delta\right]}
	\stackrel{(a)}{=}(\delta\lambda)^{K-K_0}\cdot \frac{C_2^{K}}{C_1^{K_0}}\cdot\frac{\prod_{i=1}^{K}\left|L_{V_{i}}+n_i^{-1}J\right|}{\prod_{j=1}^{K_0}\left|L_{V_{j}^0}+n_{0,j}^{-1}J\right|}\\
	=&(\delta\lambda)^{K-K_0}\cdot \frac{C_2^{K}}{C_1^{K_0}}\cdot\left(\prod_{i=1}^K\frac{|L_{V_i}+n_i^{-1}J|}{\prod_{j=1}^{K_0}|L_{W_{ij}}+m_{ij}^{-1}J|}\right)\cdot\left(\prod_{j=1}^{K_0}\frac{\prod_{i=1}^{K}|L_{W_{ij}}+m_{ij}^{-1}J|}{|L_{V_{j}^0}+n_{0,j}^{-1}J|}\right),
\end{aligned}
\]
where notationally $|L_{W_{ij}}+m_{ij}^{-1}J|\stackrel{\Delta}{=}1$ if $W_{ij}$ is an empty set, and $(a)$ is due to (9) in the main paper.

First, we consider $K_0=1$; that is, $(V^0_1=\{1,\ldots,n\})$ is the oracle clustering. By \eqref{ineq:upperbound},
\(
\begin{aligned}
\frac{\Pi^*(\mathcal V_n\not=(V^0_1)|y)}{\Pi(\mathcal V_n\sim(V^0_1)|y)}&= \sum_{K=2}^n\sum_{(V_1,\ldots,V_K)\in \mathcal V^{K,n}}\frac{\Pi^*(\mathcal V_n=(V_1,\ldots,V_K)|y)}{1\cdot\Pi^*(\mathcal V_n=(V^0_1)|y)}\\
&\stackrel{(a)}{\leq} \sum_{K=2}^nK^n\max_{(V_1,\ldots,V_K)\in\mathcal V^{K,n}}\frac{\Pi^*(\mathcal V_n=(V_1,\ldots,V_K)|y)}{\Pi^*(\mathcal V_n=(V^0_1)|y)}\\
&\stackrel{(b)}{\leq} \sum_{K=2}^n K^n(\delta\lambda)^{K-1}\cdot\frac{C_2^K}{C_1}\cdot \frac{\prod_{i=1}^K|L_{V_i}+n_i^{-1}J|}{|L_n+n^{-1}J|}\\
&\stackrel{(c)}{\leq} \sum_{K=2}^n K^n(\delta\lambda)^{K-1}\cdot\frac{C_2^K}{C_1}\cdot \left(\frac{1}{n\gamma}\right)^{K-1}\\
&\stackrel{(d)}{\precsim} \sum_{K=2}^n \exp\left\{-[(K-1)\log(2+\iota_2)-\log(K)]n+(K-1)\log\left(\frac{C_2}{n}\right)\right\}\\
&\precsim \sum_{K=2}^n \exp\left\{-[(K-1)\log(2+\iota_2)-\log(K)]n\right\}\\
&\stackrel{(e)}{\leq} \frac{n-1}{(1+\iota_2/2)^n}\to0\text{ as }n\to\infty,
\end{aligned}
\)
where $(a)$ uses the fact that the total number of ways of assigning $n$ points into $K$ clusters is no greater than $K^n$; $(b)$ uses \eqref{ineq:upperbound}; $(c)$ invokes Lemma~5.7 in the main paper; $(d)$ is by the second condition of $\mathcal D^{(\infty)}$; $(e)$ uses the fact that the function $g_1(K)=(K-1)\log(2+\iota_2)-\log(K)$ is an increasing function in $K$. 

In the following, we consider $K_0>1$. We will obtain bounds for the two products in the last line of \eqref{ineq:upperbound}. Since $y^{(\infty)}\in\mathcal D^{(\infty)}$, we have
\[
\label{ineq:ratio_eps_gamma}	
\frac{\varepsilon}{\gamma}=\frac{\varepsilon}{\delta\lambda}\frac{\delta\lambda}{\gamma}\precsim [(K_0-1+\iota_1)(K_0+1+\iota_2)]^{-n}.
\]
Since the right-hand side tends to zero as $n\to \infty$, we must have $\varepsilon<\gamma$ for $n$ sufficiently large.

	For the first product in the last line of \eqref{ineq:upperbound}, we observe that for the refinement of $V_i$, $(W_{ij})_{j\in[K_0]}$, it holds for sufficiently large $n$ that
	\(
 \begin{aligned}
	\sup_{s\not\sim t\text{ under }(W_{ij})_{j\in[K_0]}}f^{(n)}_{st}\leq& \sup_{s\not\sim t\text{ under }(V^0_{j})_{j\in[K_0]}}f^{(n)}_{st}\leq \varepsilon\\
 &<\gamma\leq\inf_{s'\sim t'\text{ under }(V^0_{j})_{j\in[K_0]}} f^{(n)}_{s't'}\leq\inf_{s'\sim t'\text{ under }(W_{ij})_{j\in[K_0]}} f^{(n)}_{s't'}.
 \end{aligned}
	\)
	Thus we can apply Lemma~5.5 in the main paper and get
	\[\label{ineq:use_lemma_1}
	\prod_{i=1}^K\frac{|L_{V_i}+n_i^{-1}J|}{\prod_{j=1}^{K_0}|L_{W_{ij}}+m_{ij}^{-1}J|}\leq\prod_{i=1}^K\left\{(n_i\varepsilon)^{a_i-1}\prod_{j=1}^{K_0}\prod_{k=2}^{m_{ij}}\left(1+\frac{(n_i-m_{ij})\varepsilon}{\lambda_{ijk}}\right)\right\},
	\]
	where we let $\prod_{k=2}^0\stackrel{\Delta}{=}\prod_{k=2}^1\stackrel{\Delta}{=}1$.

	To bound $\lambda_{ijk}$, we first note that $L_{W_{ij}}=L_{\gamma}+L_\Delta$ where $L_\gamma$ is a Laplacian generated by $n$ nodes with edge weights all equal to $\gamma$ and $L_{\Delta}=L_{W_{ij}}-L_\gamma$ is a nonnegative definite matrix due to the fact that $\inf_{s,t\in W_{ij}}f^{(n)}_{st}\geq\gamma$. Due to Weyl's Inequality, it holds that
	\[
	\label{ineq:lambda_ijk}
	\lambda_{ijk}\geq\lambda_k(L_{\gamma})+\lambda_{\min}(L_\Delta)\geq\lambda_k(L_\gamma)=m_{ij}\gamma,\quad\forall 2\leq k\leq m_{ij}.
	\]
	
	It follows that
	\(
		\begin{aligned}
			&\prod_{i=1}^K\prod_{j=1}^{K_0}\prod_{k=2}^{m_{ij}}\left[1+\frac{(n_i-m_{ij})\varepsilon}{\lambda_{ijk}}\right]\\&\stackrel{(a)}{\leq}\prod_{i=1}^K\prod_{j=1}^{K_0}\prod_{k=2}^{m_{ij}}\left[1+\frac{(n_i-m_{ij})\varepsilon}{m_{ij}\gamma}\right]
			=\prod_{i=1}^K\prod_{j=1}^{K_0}\left[1+\frac{(n_i-m_{ij})\varepsilon}{m_{ij}\gamma}\right]^{m_{ij}-1}\\
			&\stackrel{(b)}{\leq}\exp\left\{\sum_{i=1}^K\sum_{j=1}^{K_0}(m_{ij}-1)\frac{(n_i-m_{ij})\varepsilon}{m_{ij}\gamma}\right\}
			<\exp\left\{\frac{\varepsilon}{\gamma}\sum_{i=1}^K\sum_{j=1}^{K_0}(n_i-m_{ij})\right\}\\
			&\leq\exp\left\{(K_0-1)\frac{n\varepsilon}{\gamma}\right\},
		\end{aligned}	
	\)
	where $(a)$ uses \eqref{ineq:lambda_ijk} and $(b)$ is due to the fact that $(1+x)\leq \exp(x)$ for $x>-1$.

	Hence, we have
	\[
	\label{ineq:first_prod}
 \begin{aligned}
	\prod_{i=1}^K\frac{|L_{V_i}+n_i^{-1}J|}{\prod_{j=1}^{K_0}|L_{W_{ij}}+m_{ij}^{-1}J|}\leq&\prod_{i=1}^K(n_i\varepsilon)^{a_i-1}\cdot\exp\left\{(K_0-1)\frac{n\varepsilon}{\gamma}\right\}\\
 \leq&(n\varepsilon)^{K^*-K}\cdot\exp\left\{(K_0-1)\frac{n\varepsilon}{\gamma}\right\}.
 \end{aligned}
	\]
	For the second product in the last line of \eqref{ineq:upperbound}, it follows from Lemma~5.7 in the main paper that
	\[
	\label{ineq:second_prod}
	\prod_{j=1}^{K_0}\frac{\prod_{i=1}^{K}|L_{W_{ij}}+m_{ij}^{-1}J|}{|L_{V_{0,j}}+n_{0,j}^{-1}J|}\leq\prod_{j=1}^{K_0}\left(\frac{1}{n_{0,j}\gamma}\right)^{b_j-1}\leq\prod_{j=1}^{K_0}\frac{1}{\gamma^{b_j-1}}=\frac{1}{\gamma^{K^*-K_0}}.
	\]
	
	Combining \eqref{ineq:upperbound}, \eqref{ineq:first_prod}, and \eqref{ineq:second_prod}, we have
	\[
		\label{ineq:single_ratio}
		\begin{aligned}
			&\frac{\Pi^*(\mathcal V_n=(V_1,\ldots,V_K)|y)}{\Pi^*(\mathcal V_n=(V_{1}^0,\ldots,V_{K_0}^0)|y)}\\
   \leq&(\delta\lambda)^{K-K_0}\cdot\frac{C_2^K}{C_1^{K_0}}\cdot\frac{(n\varepsilon)^{K^*-K}}{\gamma^{K^*-K_0}}\exp\left\{(K_0-1)\frac{n\varepsilon}{\gamma}\right\}\\
			\precsim& (\delta\lambda)^{K-K_0}\cdot\frac{C_2^K}{C_1^{K_0}}\cdot\frac{(n\varepsilon)^{K^*-K}}{\gamma^{K^*-K_0}}\\
			=&\exp\left\{(K-K_0)\log\left(\frac{\delta\lambda}{\gamma}\right)+(K^*-K)\log\left(\frac{n\varepsilon}{\gamma}\right)+K\log(C_2)-K_0\log(C_1)\right\},
		\end{aligned}
	\]
	where {the constant omitted in the $\precsim$ step above can be $2$ because as $n\to\infty$,
	\[\label{ineq:converge}
	(K_0-1)n\varepsilon/\gamma\precsim n^{3/2}(1+\iota_1)^{-n}\to0,\qquad \exp\left\{(K_0-1)\frac{n\varepsilon}{\gamma}\right\}\to 1.
	\]}

	Next, we focus on the following summations:
	\[
		\label{eq:three_sums}
		\begin{aligned}
			&\frac{\Pi(\mathcal V_n\not\sim(V^0_{1},\ldots,V^0_{K_0})|y)}{\Pi(\mathcal V_n\sim(V^0_{1},\ldots,V^0_{K_0})|y)}\\
   =&\left(\sum_{\mathcal V_n:K<K_0}+\sum_{\mathcal V_n:K=K_0,\atop\mathcal V_n\not\sim(V^0_{1},\ldots,V^0_{K_0})}+\sum_{\mathcal V_n:K>K_0}\right)\frac{\Pi^*(\mathcal V_n=(V_1,\ldots,V_K)|y)}{K_0!\cdot\Pi^*(\mathcal V_n=(V^0_{1},\ldots,V^0_{K_0})|y)}.
		\end{aligned}	
	\]
	We denote these summations by $S_1$, $S_2$, and $S_3$, respectively. 
    First,
\[
\label{ineq:S_1}
\begin{aligned}
S_1&\stackrel{(a)}{\leq}\sum_{K=1}^{K_0-1} \frac{K^n}{K_0!} \max_{(V_1,\ldots,V_K)\in\mathcal V^{n,K}} \frac{\Pi^*(\mathcal V_n=(V_1,\ldots,V_K)|y)}{\Pi^*(\mathcal V_n=(V^0_{1},\ldots,V^0_{K_0})|y)}\\
&\stackrel{(b)}{\precsim}\sum_{K=1}^{K_0-1}\exp\bigg\{(K_0-K)\log\left(\frac{\varepsilon}{\delta\lambda}\right)+(K^*-K_0)\log\left(\frac{n\varepsilon}{\gamma}\right)+(K_0-K)\log(n)\\
&\qquad\qquad\qquad+K\log(C_2)-K_0\log(C_1)+n\log(K)-\log(K_0!)\bigg\}\\
&\stackrel{(c)}{\precsim}\sum_{K=1}^{K_0-1}\exp\left\{(K_0-K)\log\left(\frac{\varepsilon}{\delta\lambda}\right)+(K^*-K_0)\log\left(\frac{n\varepsilon}{\gamma}\right)+2K_0\log(n)+n\log(K)\right\}\\
&\stackrel{(d)}{\precsim} \sum_{K=1}^{K_0-1}\exp\left\{(K_0-K)\log\left(\frac{\varepsilon}{\delta\lambda}\right)+2K_0\log(n)+n\log(K)\right\}\\
&\stackrel{(e)}{\precsim}\sum_{K=1}^{K_0-1}\exp\left\{-[(K_0-K)\log(K_0-1+\iota_1)-\log(K)]n+2K_0\log(n)\right\}\\
&\stackrel{(f)}{\le}(K_0-1)\exp\left\{-n\log\left(\frac{K_0-1+\iota_1}{K_0-1}\right)+2K_0\log(n)\right\}\to0\text{ as }n\to\infty,
\end{aligned}
\]
where $(a)$ is due to that the total number of ways of assigning $n$ points into $K$ clusters is $\le K^n$; $(b)$ uses \eqref{ineq:single_ratio}; in $(c)$, we upperbound the rest  by $K_0\log(n)$; $(d)$ is due to $K^*\geq K_0$ and $\log(\sfrac{n\varepsilon}{\gamma})\leq 0$ for sufficiently large $n$ ; $(e)$ uses the first condition of $\mathcal D^{(\infty)}$ which will contribute some fixed constants by definition; $(f)$ is because $g_2(K)=(K_0-K)\log(1+\iota_1)-\log(K)$ is decreasing for $K\geq1$. If $K_0$ is bounded the convergence holds trivially; otherwise, we have $[(K_0-1+\iota_1)/(K_0-1)]^{-n}\asymp \exp(-\iota_1n/(K_0-1))$ and $2K_0\log(n)=o(n/(K_0-1))$ due to Assumption~2 in the main paper, which lead to the convergence. { The constants omitted in $\precsim$ steps can be $2$, $1$, and $1$, respectively, in steps $(b)$, $(c)$, and $(d)$ for sufficiently large $n$.}

Second, following similar ideas, we have {for sufficiently large $n$,
\[
\label{ineq:S_2}
\begin{aligned} 
	S_2&\leq \frac{K_0^n}{K_0!}\max_{(V_1,\ldots,V_{K_0})\in\mathcal V^{n,K_0}\atop (V_1,\ldots,V_{K_0})\not\sim(V^0_1,\ldots,V^0_{K_0})} \frac{\Pi^*(\mathcal V_n=(V_1,\ldots,V_K)|y)}{\Pi^*(\mathcal V_n=(V_{1}^0,\ldots,V^0_{K_0})|y)}\\
	&\le 2\exp\left\{(K^*-K_0)\log\left(\frac{n\varepsilon}{\gamma}\right)+K_0\log(C_2)-K_0\log(C_1)+n\log(K_0)-\log(K_0!)\right\}\\
	&\le 2C^{(2)} \left(\frac{n\varepsilon}{\gamma}\right)^{K^*-K_0}K_0^n\\
	&\stackrel{(b)}{\le} 2C^{(2)}\left(\frac{n\varepsilon}{\gamma}\right)K_0^n\\
	&\precsim 2C^{(2)}n\left(\frac{K_0}{(K_0-1+\iota_1)(K_0+1+\iota_2)}\right)^{n}\to0\text{ as }n\to\infty,\\
	\end{aligned}	
\]
where $C^{(2)}=\max_{t\in\mathbb N^+}(C_2/C_1)^t/t!$, and $(b)$ uses $n\varepsilon/\gamma\to0$ as $n\to\infty$, and the fact that if $|\mathcal V_n|=K_0$, then $K^*=K_0$ if and only if $\mathcal V_n\sim(V_{0,1},\ldots,V_{0,K_0})$. The last $\precsim$ step is due to \eqref{ineq:ratio_eps_gamma} and only some fixed constant is induced. 

Third, following similar ideas, we have for sufficiently large $n$,
\[
\label{ineq:S_3}
\begin{aligned}
	S_3&\precsim \sum_{K=K_0+1}^n\exp\left\{(K-K_0)\log\left(\frac{\delta\lambda}{\gamma}\right)+(K^*-K)\log\left(\frac{n\varepsilon}{\gamma}\right)+K\log(C_2)+n\log(K)\right\}\\
	&\leq\sum_{K=K_0+1}^n \exp\left\{(K-K_0)\log\left(\frac{\delta\lambda}{\gamma}\right)+K\log(C_2)+n\log(K)\right\}\\
	&\stackrel{(a)}{\precsim} \sum_{K=K_0+1}^n \exp\left\{-[(K-K_0)\log(K_0+1+\iota_2)-\log(K)]n+K\log(C_2)\right\}\\
	&\stackrel{(b)}{\le} (n-K_0)\exp\left\{-n\log\left(\frac{K_0+1+\iota_2}{K_0+1}\right)+(K_0+1)\log(C_2)\right\}\to0\text{ as }n\to\infty,
\end{aligned}
\]
where the first two steps follow the similar idea as in the first few steps in \eqref{ineq:S_1} and \eqref{ineq:S_2}, $(a)$ uses the second condition of $\mathcal D^{(\infty)}$ which will contribute some fixed constants by definition, and $(b)$ is due to the following argument. 
Define the function $g_3(K)=-[(K-K_0)\log(K_0+1+\iota_2)-\log(K)]n+K\log(C_2)$. It has derivative $g_3'(K)=-[\log(K_0+1+\iota_2)-1/K]n+\log(C_2)\leq0$ for $K\geq K_0+1$ and sufficiently large $n$. Hence, $g_3(K)\leq g_3(K_0+1)=-n\log\left(\sfrac{(K_0+1+\iota_2)}{(K_0+1)}\right)+(K_0+1)\log(C_2)$. The convergence to zero holds due to Assumption~2 in the main paper.}

Combining \eqref{eq:three_sums}, \eqref{ineq:S_1}, \eqref{ineq:S_2}, and \eqref{ineq:S_3}, we have
\(
\begin{aligned}
	\Pi(\mathcal V_n\not\sim(V_{0,1},\ldots,V_{0,K_0})|y)&=\Pi(\mathcal V_n\sim(V_{0,1},\ldots,V_{0,K_0})|y)(S_1+S_2+S_3)\to0.	
\end{aligned}
\)

\end{proof}

\begin{proof}[Proof of Lemma~3.5]

\(
\begin{aligned}
P_0^{(n)}\left(y^{(n)}\not\in\mathcal D^{(n)}_\phi\right)
&=P_0^{(n)}\left(\left(\bigcup_{s\not\sim t;s,t\in[n]}\left\{d_{st}^2< a_n\right\}\right)\bigcup\left(\bigcup_{s'\sim t';s',t'\in[n]}\left\{d_{s't'}^2> b_n\right\}\right)\right)\\
&\quad\leq\sum_{s\not\sim t;s,t\in[n]}P_0^{(n)}(d^2_{st}<a_n)+\sum_{s'\sim t';s',t'\in[n]}P_0^{(n)}(d^2_{s't'}>b_n)\\
&\quad\leq\left(\sum_{1\leq i<j\leq K_{0,n}}n_in_j\right)\cdot\max_{s\not\sim t;s,t\in[n]}P_0^{(n)}(d^2_{st}<a_n)\\
&\qquad\qquad\qquad+\left(\sum_{k=1}^{K_{0,n}}{n_k\choose 2}\right)\cdot\max_{s'\sim t';s',t'\in[n]}P_0^{(n)}(d^2_{s't'}>b_n)\\
&\quad\leq n^2\cdot\max_{s\not\sim t;s,t\in[n]}P_0^{(n)}(d^2_{st}<a_n)+n^2\cdot\max_{s'\sim t';s',t'\in[n]}P_0^{(n)}(d^2_{s't'}>b_n)\\
&\quad=n^2\cdot\sup_{k\not=\ell;k,\ell\in[K_{0,n}]}P(D_{k\ell}^2<a_n)+n^2\cdot\sup_{k'\in[K_{0,n}]}P(D_{k'k'}^2>b_n).
\end{aligned}
\)
\end{proof}
{ This proves the required condition for clustering consistency as established in Theorem 3.4.}

\begin{proof}[Proof of Lemma~3.6]
\[
\label{ineq:E_misclass_rate}
&\mathbb{E}(d(\mathcal V_n,(V_1^{0,n},\ldots,V_{K_0}^{0,n}))\mid y^{(n)})\\
=&\Pi^*(\mathcal V_n=(V_1^{0,n},\ldots,V_{K_0}^{0,n}))|y^{(n)})\cdot\sum_{r=1}^n\sum_{\mathcal V^*_n: d(\mathcal V^*_n,(V_1^{0,n},\ldots,V_{K_0}^{0,n}))=r} \frac{r\Pi^*(\mathcal V_n=\mathcal V^*_n \mid y^{(n)})}{\Pi^*(\mathcal V_n=(V_1^{0,n},\ldots,V_{K_0}^{0,n})|y^{(n)})}\\
\le&\sum_{r=1}^n\sum_{\mathcal V^*_n: d(\mathcal V^*_n,(V_1^{0,n},\ldots,V_{K_0}^{0,n}))=r} \frac{r\Pi^*(\mathcal V_n=\mathcal V^*_n \mid y^{(n)})}{\Pi^*(\mathcal V_n=(V_1^{0,n},\ldots,V_{K_0}^{0,n})|y^{(n)})}\\
\le& nK_0^n\max_{\mathcal V_n^*:|\mathcal V^*_n|=K_0,\mathcal V_n^*\not\sim(V_1^{0,n},\ldots,V_{K_0}^{0,n})}\frac{\Pi^*(\mathcal V_n=\mathcal V^*_n \mid y^{(n)})}{\Pi^*(\mathcal V_n=(V_1^{0,n},\ldots,V_{K_0}^{0,n}))|y^{(n)})}.
\]

It follows from \eqref{ineq:E_misclass_rate}, \eqref{ineq:single_ratio}, and the assumption that $\varepsilon_n/\gamma_n=o(1/n)$ that
\(
\mathbb{E}(d(\mathcal V_n,(V_1^{0,n},\ldots,V_{K_0}^{0,n}))\mid y^{(n)})\precsim&\exp\left\{\log\left(\frac{\varepsilon_n}{\gamma_n}\right)+n\log(K_0)+2\log(n)+K_0\log\left(\frac{C_2}{C_1}\right)\right\}\\
\precsim& \exp\left\{\log\left(\frac{\varepsilon_n}{\gamma_n}\right)+n\log(K_0+1)\right\}.
\)

Lastly, for the Gaussian-BSF model, plugging the specific form of $f^{(n)}_{st}$ in the above finishes the proof.
\end{proof}

\begin{proof}[Proof of Theorem~3.7]
Let $F_n:=\mathbb{E}(d_H(\mathcal V_n,(V_1^{0,n},\ldots,V_{K_{0,n}}^{0,n}))\mid y^{(n)})$, and let $\tilde{\mathcal D}:=\{y^{(n)}:\\\inf_{s\not\sim t}d^2_{st}\ge a'_n,\sup_{s'\sim t'}d^2_{s't'}\le b'_n\}$. Note that $F_n\le n$. Then
\(
\mathbb E_{P^{(n)}_0}[F_n]&=\mathbb E_{P^{(n)}_0}[F_n1_{\tilde {\mathcal D}^c}]+\mathbb E_{P^{(n)}_0}[F_n1_{\tilde{\mathcal D}}]\\
&\le nP^{(n)}_0(y^{(n)}\not\in\tilde{\mathcal D})+\mathbb E_{P^{(n)}_0}[F_n1_{\tilde{\mathcal D}}]\\
&\precsim nP^{(n)}_0(y^{(n)}\not\in\tilde{\mathcal D}) + \exp\left\{-\frac{a'_n-b'_n}{2\sigma_n^2}+n\log(K_{0,n}+1)\right\}\\
&\le n^3\left[\sup_{k\not=\ell;k,\ell\in[K_{0,n}]}P(D^2_{k\ell}<a'_n)+\sup_{k'\in[K_{0,n}]}P(D^2_{k'k'}>b'_n)\right]\\
&\quad +\exp\left\{-\frac{a'_n-b'_n}{2\sigma_n^2}+n\log(K_{0,n}+1)\right\}.
\)

\end{proof}

\begin{proof}[Proof of Theorem~4.2]

Take $\phi=(1,1,1,\iota/2)$. We have

\(
\begin{cases}
a_n=2\sigma_n^2\left[n\log(K_{0,n})+\log(\rho_n)-p\log(\sqrt{2\pi})\right],\\
b_n=2\sigma_n^2\left[-n\log(K_{0,n}+1+\iota/2)+\log(\rho_n)-p\log(\sqrt{2\pi})\right].
\end{cases}
\)

By assumptions (ii), we have
\(
a_n\asymp b_n\asymp \sigma_n^2\log(\rho_n).
\)

For $k,\ell\in[K_{0,n}]$ with $k\not=\ell$, let $\Delta_{kl}\sim\Normal(\mu_k-\mu_\ell,\Sigma_k+\Sigma_\ell)$ and $\Delta_{kk}\sim\Normal(0,2\Sigma_k)$.
Next, we prove that both (7) and (8) in the main paper hold.

\begin{itemize}
\item Since $D^2_{\mu,\min}/[\sigma^2_n\log(\rho_n)]\to\infty$ as $n\to\infty$, it follows that $0\leq\sfrac{a_n}{\|\mu_k-\mu_\ell\|^2_2}\leq\sfrac{a_n}{D^2_{\mu,\min}}\to0$ as $n\to\infty$ for any $k\not=\ell\in[K_{0,n}]$. Hence, for sufficiently large $n$, if $D^2_{k\ell}=\|\Delta_{k\ell}\|_2^2<a_n$, then $\|\Delta_{k\ell}-\mu_{k}+\mu_{\ell}\|_2^2>D^2_{\mu,\min}/2$. On the other hand, we have 
\(
\begin{aligned}
&\|\Delta_{k\ell}-\mu_{k}+\mu_{\ell}\|_2^2\\
\leq&\frac{(\Delta_{k\ell}-\mu_{k}+\mu_{\ell})^T(\Sigma_{k}+\Sigma_{\ell})^{-1}(\Delta_{k\ell}-\mu_{k}+\mu_{\ell})}{\lambda_{\min}((\Sigma_{k}+\Sigma_{\ell})^{-1})}\\
\leq&2\Lambda_{\max}(\Delta_{k\ell}-\mu_{k}+\mu_{\ell})^T(\Sigma_{k}+\Sigma_{\ell})^{-1}(\Delta_{k\ell}-\mu_{k}+\mu_{\ell}).
\end{aligned}
\)

It follows that
\[
\label{ineq:cond1}
\begin{aligned}
&\max_{k\not=\ell;k,\ell\in[K_{0,n}]}P(D^2_{k\ell}<a_n)
\leq  \max_{k\not=\ell;k,\ell\in[K_{0,n}]}P(\|\Delta_{k\ell}-\mu_k+\mu_\ell\|_2^2>D^2_{\mu,\min}/2)\\
\leq&P\left((\Delta_{k\ell}-\mu_{k}+\mu_{\ell})^T(\Sigma_{k}+\Sigma_{\ell})^{-1}(\Delta_{k\ell}-\mu_{k}+\mu_{\ell})>\frac{D^2_{\mu,\min}}{4\Lambda_{\max}}\right)\\
\stackrel{(a)}{\leq}&\exp\left\{-\frac{D^2_{\mu,\min}}{8\Lambda_{\max}}+\frac{p}{2}\log\left(\frac{D^2_{\mu,\min}}{8p\Lambda_{\max}}\right)-\frac{p}{2}\right\}\\
\stackrel{(b)}{\precsim}&\exp\left\{-\frac{D^2_{\mu,\min}}{16\Lambda_{\max}}\right\}\\
\stackrel{(c)}{\precsim}&1/n^3,
\end{aligned}
\]
where $(a)$ is due to Lemma~4.1 in the main paper and $D^2_{\mu,\min}/p\Lambda_{\max}\to\infty$, and $(b),(c)$ are due to the assumption that $D^2_{\mu,\min}/\Lambda_{\max}\log(n)\to\infty$ as $n\to\infty$.

\item Following a similar idea as above, we have
\[
\label{ineq:cond2}
\begin{aligned}
&\max_{k'\in[K_{0,n}]}P(D^2_{k'k'}>b_n)\\
\leq&\max_{k'\in[K_{0,n}]}P\left(\Delta_{k'k'}^T(2\Sigma_{k'})^{-1}\Delta_{k'k'}>\frac{b_n}{2\Lambda_{\max}}\right)\\
\leq&\exp\left\{-\frac{b_n}{4\Lambda_{\max}}+\frac{p}{2}\log\left(\frac{b_n}{2p\Lambda_{\max}}\right)-\frac{p}{2}\right\}\\
\precsim&\exp\left\{-\frac{\sigma^2_n\log(\rho_n)}{8\Lambda_{\max}}\right\}\\
\precsim&1/n^3.
\end{aligned}
\]

\end{itemize}

Combining \eqref{ineq:cond1} and \eqref{ineq:cond2} and invoking Theorem~3.5 in the main paper finishes the proof.

\end{proof}

\begin{proof}[Proof of Theorem~4.4]
Take $\phi=(1,1,1,\iota/2)$. We have 
\(
\begin{cases}
a_n=2\sigma_n^2\left[n\log(K_{0,n})+\log(\delta_n\lambda_n\rho_n)\right],\\
b_n=2\sigma_n^2\left[-n\log(K_{0,n}+1+\iota/2)+\log(\delta_n\lambda_n\rho_n)\right].
\end{cases}
\)

Following the same argument in the proof of Theorem~4.2, we have
\(
a_n\asymp b_n\asymp \sigma_n^2\log(\rho_n),
\)
and for any $k,\ell\in[K_{0,n}]$,
\[
\label{ineq:a_n_slower_D}
0\leq\frac{\sqrt{a_n}}{d(\mu_k,\mu_\ell)}\leq\frac{\sqrt{a_n}}{D_{\mu,\min}}\asymp \sqrt{\frac{\sigma_n^2\log(\rho_n)}{D^2_{\mu,\min}}}\to0\text{ as }n\to\infty.
\]

\begin{itemize}
\item First, consider $P(D^2_{k\ell}<a_n)$ for $k,\ell\in[K_{0,n}]$ with $k\not=\ell$. Let $X_k\stackrel{\text{indep}}{\sim}G_k^0$ for $k\in[K_{0,n}]$.
It is not hard to see that
\(
d(X_k,X_\ell)\geq d(\mu_{k},\mu_{\ell}) - d(\mu_{k},X_k)-d(\mu_{\ell},X_\ell).
\)

Hence,
\(
\begin{aligned}
&P(D^2_{k\ell}< a_n)\\
\leq& P(d(\mu_{k},\mu_{\ell}) - d(\mu_{k},X_k)-d(\mu_{\ell},X_\ell) < \sqrt{a_n}) \\
\leq& P\left(d(\mu_{k},X_k)> \frac{d(\mu_{k},\mu_{\ell})-\sqrt{a_n}}{2}\right)+P\left(d(\mu_{\ell},X_\ell)> \frac{d(\mu_{k},\mu_{\ell})-\sqrt{a_n}}{2}\right)\\
\leq& 2\max_{k'\in[K_{0,n}]}P\left(d(X_{k'},\mu_{k'})> \frac{D_{\mu,\min}-\sqrt{a_n}}{2}\right).
\end{aligned}
\)

Due to \eqref{ineq:a_n_slower_D},  we have for sufficiently large $n$, $(D_{\mu,\min}-\sqrt{a_n})/2\geq D_{\mu,\min}/4$, which leads to
\(
\begin{aligned}
P(D^2_{k\ell}< a_n)\leq& 2\max_{k'\in[K_{0,n}]}P\left(d(X_{k'},\mu_{k'})> \frac{D_{\mu,\min}}{4}\right)\leq2\exp\left\{-C\left(\frac{D_{\mu,\min}}{4}\right)^{\nu}\right\}.
\end{aligned}
\)

Due to Assumptions (v) and (vi), we have $D_{\mu,\min}^\nu/\log(n)\to\infty$, so
\[
\label{ineq:cond1_general}
P(D^2_{k\ell}<a_n)\precsim 1/n^3.
\]

\item Second, consider $P(D_{k'k'}^2>b_n)$ for $k'\in[K_{0,n}]$. We have
\(
&P(D_{k'k'}^2>b_n)\leq 2P(d(X_{k'},\mu_{k'})>\sqrt{b_n}/2)\leq 2\exp\left\{-C(\sqrt{b_n}/2)^\nu\right\}.
\)
By Assumption (v), we have $b_n^{\nu/2}/\log(n)\to\infty$, so
\[
\label{ineq:cond2_general}
P(D^2_{k'k'}>b_n)\precsim 1/n^3.
\]

\end{itemize}

Combining \eqref{ineq:cond1_general} and \eqref{ineq:cond2_general} and invoking Theorem~3.5 finishes the proof.

\end{proof}

\begin{proof}[Proof of Theorem~4.5]
Take $a'_n=D^2_{\mu,\min}/2$ and $b'_n=D^2_{\mu,\min}/4$.
Following the similar idea in \eqref{ineq:cond1} and \eqref{ineq:cond2}, we have
\(
\sup_{k\not=\ell;k,\ell\in[K_{0,n}]}P(D^2_{k\ell}<a'_n)+\sup_{k'\in[K_{0,n}]}P(D^2_{k'k'}>b'_n)\precsim \exp(-O(1)\cdot SNR^2).
\)
Since $SNR/\sqrt{\log(n)}\to\infty$, multiplying the above by $n^3$ yields the asymptotic upper bound $\exp(-O(1)\cdot SNR^2)$.

On the other hand, in light of the rate condition on $\sigma_n^2$, we have 
\(
\exp\left\{-\frac{a'_n-b'_n}{2\sigma_n^2}+n\log(K_{0,n}+1)\right\}&=\exp\left\{-O(1)\cdot\frac{D^2_{\mu,\min}}{\sigma_n^2}+o(1)\cdot\frac{D^2_{\mu,\min}}{\sigma_n^2}\right\}\\
&=\exp\left\{-O(1)\cdot\frac{\Lambda_{\max}}{\sigma_n^2}\cdot SNR^2\right\}\\
&\precsim\exp(-O(1)\cdot SNR^2).
\)
Invoking Theorem~3.7 finishes the proof.

\end{proof}

{
\section{Simulation Study}
\label{sec:simulation}

We conduct two numerical experiments to illustrate the finite-sample
behavior of the BSF clustering procedure and to characterize the
sample sizes at which the asymptotic guarantees become empirically
operative. The BSF posterior mode is computed following \cite{duan2023spectral},
which also provides guidelines for hyperparameter specification.
Each experiment uses $R = 50$ independent Monte Carlo replications.

\subsection{Convergence of misclassification rate}
\label{sec:expA}

We generate $n$ observations from two balanced Gaussian clusters
($K_0 = 2$, $n_k = n/2$) in one dimension with unit noise variance.
The separation $D_{\mu,{\min}}$ varies across three regimes:
\begin{enumerate}
  \item \textit{Easy:} $D_{\mu,{\min}} = n^{0.4}$, so
        $\mathrm{SNR}/\sqrt{\log n} \to \infty$.
        Corollary~4.3 guarantees consistency.
  \item \textit{Medium-hard:} $D_{\mu,{\min}} = (\log n)^{0.6}$, which grows
        but only marginally satisfies the threshold condition.
  \item \textit{Hard:} $D_{\mu,{\min}} = (\log n)^{0.4}$, at the lower edge
        of the identifiable regime; no polynomial rate of convergence
        is guaranteed.
\end{enumerate}
Sample sizes range over
$n \in \{20, 30, 40, 50, 60, 75, 100, 120, 150, 200, 300, 500\}$.

Table~\ref{tab:expA_main} reports mean misclassification rates at
$n \in \{60, 120, 300\}$; $\hat{P}(\hat{K}=K_0)=1.00$ throughout.
Table~\ref{tab:expA_onset} gives the full profile across all sample sizes.

In the easy regime, the rate drops from $0.054$ at $n=20$ to
essentially zero by $n=200$, consistent with rapid exponential
decay; the asymptotic regime is already operative at
$n \approx 50$, where the rate first falls below $1\%$.
In the medium-hard regime, convergence is slower: the rate is $17\%$
at $n=20$ and reaches $8\%$ only at $n=500$, with a clear downward
trend visible from $n \approx 75$ onward.
In the hard regime, the rate decreases very slowly, from $22\%$ at
$n=20$ to $16\%$ at $n=300$.

\begin{table}[ht]
  \centering
  \caption{Mean misclassification rate for $K_0=2$ balanced Gaussian
    clusters, based on 50 replications.}
  \label{tab:expA_main}
  \begin{tabular}{lccc}
    \hline
    Regime & $n=60$ & $n=120$ & $n=300$ \\
    \hline
    Easy        & 0.008 & 0.001 & 0.000 \\
    Medium-hard & 0.135 & 0.103 & 0.077 \\
    Hard        & 0.195 & 0.182 & 0.157 \\
    \hline
  \end{tabular}
\end{table}

\begin{table}[ht]
  \centering
  \caption{Mean misclassification rate across a fine grid of sample
    sizes, based on 50 replications.}
  \label{tab:expA_onset}
  \setlength{\tabcolsep}{7pt}
  \begin{tabular}{rccc}
    \hline
    $n$ & Easy & Medium-hard & Hard \\
    \hline
    20  & 0.054 & 0.173 & 0.219 \\
    30  & 0.035 & 0.169 & 0.218 \\
    40  & 0.019 & 0.145 & 0.211 \\
    50  & 0.009 & 0.134 & 0.196 \\
    60  & 0.008 & 0.135 & 0.195 \\
    75  & 0.004 & 0.108 & 0.180 \\
    100 & 0.001 & 0.107 & 0.182 \\
    120 & 0.001 & 0.103 & 0.182 \\
    150 & 0.000 & 0.096 & 0.166 \\
    200 & 0.000 & 0.087 & 0.165 \\
    300 & 0.000 & 0.077 & 0.157 \\
    500 & 0.000 & 0.067 & 0.149 \\
    \hline
  \end{tabular}
\end{table}

\subsection{Misclassification rate as a function of $\mathrm{SNR}$ and sample size}
\label{sec:expB}

\noindent\textit{Setup.}
With $n=200$ and $K_0=2$ fixed, we vary
$\Delta \in \{1,2,3,4,5,6\}$,
giving $\mathrm{SNR}^2 \in \{1,4,9,16,25,36\}$.
Theorem~4.5 predicts
\[
  \mathbb{E}_{P_0^{(n)}}\!\bigl[
    d_H(\hat{\mathcal{V}}_n,\,\mathcal{V}_n^0)
  \bigr]
  \;\precsim\;
  \exp\!\bigl(-O(1)\cdot\mathrm{SNR}^2\bigr),
\]
implying log-linear decay in $\mathrm{SNR}^2$.
To examine the onset of this behavior at fixed signal strength, we additionally run
across $n \in \{30,50,75,100,150,200,300\}$ at two representative
signal levels ($\Delta=2$ and $\Delta=4$).

Table~\ref{tab:expB_main} shows the rate dropping from $0.307$ at
$\mathrm{SNR}^2=1$ to $0.001$ at $\mathrm{SNR}^2=36$, spanning
nearly three orders of magnitude.
A linear regression of $\log(\text{mc})$ on $\mathrm{SNR}^2$
yields a strong fit, corroborating the exponential bound of
Theorem~4.5.

\begin{table}[ht]
  \centering
  \caption{Mean misclassification rate as a function of
    $\mathrm{SNR}^2$ for $n=200$, $K_0=2$, and 50 replications.}
  \label{tab:expB_main}
  \begin{tabular}{ccc}
    \hline
    $\Delta$ & $\mathrm{SNR}^2$ & Mean misclassification rate \\
    \hline
    1 &  1 & 0.307 \\
    2 &  4 & 0.158 \\
    3 &  9 & 0.066 \\
    4 & 16 & 0.023 \\
    5 & 25 & 0.006 \\
    6 & 36 & 0.001 \\
    \hline
  \end{tabular}
\end{table}

Table~\ref{tab:expB_onset} reports onset behavior as a function of $n$
at two fixed signal levels.
At $D_{\mu,{\min}}=4$ ($\mathrm{SNR}^2=16$), the misclassification rate is
$3\%$ at $n=30$ and remains stable at $2$--$3\%$ across all larger
$n$, indicating that the method reaches its optimal level at small
sample sizes when the signal is moderately strong.
At $D_{\mu,{\min}}=2$ ($\mathrm{SNR}^2=4$), the rate is approximately
$16$--$18\%$ throughout and shows no systematic decrease with $n$ beyond $n=50$.
This is consistent with theory: for $D_{\mu,{\min}}=2$, the Bayes error for
two unit-variance Gaussians is $\Phi(-1)\approx 16\%$, and the
method is already near-Bayes-optimal at $n=50$.

\begin{table}[ht]
  \centering
  \caption{Mean misclassification rate at two fixed signal levels,
    $K_0=2$, with 50 replications.}
  \label{tab:expB_onset}
  \begin{tabular}{rcc}
    \hline
    & \multicolumn{2}{c}{Mean misclassification rate} \\
    $n$ & $D_{\mu,{\min}}=2\ (\mathrm{SNR}^2=4)$
        & $D_{\mu,{\min}}=4\ (\mathrm{SNR}^2=16)$ \\
    \hline
     30 & 0.178 & 0.032 \\
     50 & 0.160 & 0.023 \\
     75 & 0.152 & 0.022 \\
    100 & 0.160 & 0.023 \\
    150 & 0.156 & 0.025 \\
    200 & 0.158 & 0.023 \\
    300 & 0.158 & 0.023 \\
    \hline
  \end{tabular}
\end{table}
}

\bibliographystyle{imsart-nameyear} 
\bibliography{references}

\end{document}